\documentclass[oneside,english,reqno]{amsart}
\usepackage{lmodern}
\usepackage[T1]{fontenc}
\usepackage[latin9]{inputenc}
\usepackage{babel}
\usepackage{float}
\usepackage{amsthm}
\usepackage{amssymb}
\usepackage[unicode=true,pdfusetitle,
 bookmarks=true,bookmarksnumbered=false,bookmarksopen=false,
 breaklinks=false,pdfborder={0 0 1},backref=false,colorlinks=false]
 {hyperref}

\makeatletter

\floatstyle{ruled}
\newfloat{algorithm}{tbp}{loa}
\providecommand{\algorithmname}{Algorithm}
\floatname{algorithm}{\protect\algorithmname}

\numberwithin{equation}{section}
\numberwithin{figure}{section}
\theoremstyle{plain}
\newtheorem{thm}{\protect\theoremname}[section]
  \theoremstyle{plain}
  \newtheorem{lyxalgorithm}[thm]{\protect\algorithmname}
  \theoremstyle{remark}
  \newtheorem{rem}[thm]{\protect\remarkname}
  \theoremstyle{remark}
  \newtheorem{claim}[thm]{\protect\claimname}
  \theoremstyle{plain}
  \newtheorem{prop}[thm]{\protect\propositionname}
  \theoremstyle{plain}
  \newtheorem{fact}[thm]{\protect\factname}
  \theoremstyle{plain}
  \newtheorem{assumption}[thm]{\protect\assumptionname}
  \theoremstyle{plain}
  \newtheorem{lem}[thm]{\protect\lemmaname}

\makeatother

  \providecommand{\algorithmname}{Algorithm}
  \providecommand{\assumptionname}{Assumption}
  \providecommand{\claimname}{Claim}
  \providecommand{\factname}{Fact}
  \providecommand{\lemmaname}{Lemma}
  \providecommand{\propositionname}{Proposition}
  \providecommand{\remarkname}{Remark}
\providecommand{\theoremname}{Theorem}

\begin{document}
\title[Dykstra splitting and approximate proximal point]{Dykstra splitting and an approximate proximal point algorithm for minimizing the sum of convex functions} 

\begin{abstract}
We show that Dykstra's splitting for projecting onto the intersection
of convex sets can be extended to minimize the sum of convex functions
and a regularizing quadratic. We give conditions for which convergence
to the primal minimizer holds so that more than one convex function
can be minimized at a time, the convex functions are not necessarily
sampled in a cyclic manner, and the SHQP strategy for problems involving
the intersection of more than one convex set can be applied. When
the sum does not involve the regularizing quadratic, we discuss an
approximate proximal point method combined with Dykstra's splitting
to minimize this sum. 
\end{abstract}

\author{C.H. Jeffrey Pang}

\thanks{We acknowledge grant R-146-000-214-112 from the Faculty of Science,
National University of Singapore. We gratefully acknowledge discussions
with Ting-Kei Pong on Dykstra's splitting which led to this paper.}

\curraddr{Department of Mathematics\\ 
National University of Singapore\\ 
Block S17 08-11\\ 
10 Lower Kent Ridge Road\\ 
Singapore 119076 }

\email{matpchj@nus.edu.sg}

\date{\today{}}


\maketitle
\tableofcontents{}

\section{Introduction}

Throughout this paper, let $X$ be a finite dimensional Hilbert space.
Consider the problem of minimizing the sum of convex functions 
\begin{equation}
\sum_{i=1}^{r}h_{i}(\cdot),\label{eq:main-pblm}
\end{equation}
where $h_{i}:X\to\mathbb{R}\cup\{\infty\}$ are closed proper convex
functions. The aim of this paper is to combine Dykstra's splitting
and an approximate proximal point algorithm in order to minimize \eqref{eq:main-pblm}.

\subsection{Dykstra's algorithm}

For closed convex sets $C_{i}$, where $i\in\{1,\dots,r\}$, Dykstra's
algorithm \cite{Dykstra83} solves the problem 
\begin{equation}
\min_{x}\frac{1}{2}\|x-x_{0}\|^{2}+\sum_{i=1}^{r}\delta_{C_{i}}(x),\label{eq:Dykstra-form}
\end{equation}
where $\delta_{C_{i}}(\cdot)$ is the indicator function of the set
$C_{i}$. Note that \eqref{eq:Dykstra-form} is also equivalent to
the problem of projecting the point $x_{0}$ onto $\cap_{i=1}^{r}C_{i}$.
The projection onto the intersection $\cap_{i=1}^{r}C_{i}$ may be
difficult, but each step of Dykstra's algorithm requires only the
projection onto one set $C_{i}$ at a time. Its convergence to a primal
minimizer without constraint qualifications was established in \cite{BD86}.
Separately, Dykstra's algorithm was rediscovered in \cite{Han88},
who noticed that it is block coordinate minimization on the dual problem,
and proved the convergence to a primal minimizer, but under a constraint
qualification. This dual perspective was also noticed by \cite{Gaffke_Mathar},
who built on \cite{BD86} and used duality to prove the convergence
to a primal minimizer without constraint qualifications. 

Dykstra's algorithm can be made into a parallel algorithm by using
the product space approach largely attributed to \cite{Pierra84}.
But this parallelization is slower than the original Dykstra's algorithm
because the dual variables are not updated in a Gauss Seidel manner.
(In other words, the dual variables are not updated with the most
recent values of the other dual variables.) It was also noticed in
\cite{Hundal-Deutsch-97} (among other things) that the projections
onto the sets $C_{i}$ need not be performed in a cyclic manner to
achieve convergence. In \cite{Pang_DBAP}, we studied a SHQP (supporting
halfspace and quadratic programming) heuristic for improving the convergence
of Dykstra's algorithm by noticing that the projection operations
onto the sets $C_{i}$ generate halfspaces containing $C_{i}$, and
the intersection of these halfspaces can be a better approximate of
$\cap_{i=1}^{r}C_{i}$ than each $C_{i}$ alone. 

We now refer to the natural extension of Dykstra's algorithm for minimizing 

\begin{equation}
\min_{x}\frac{1}{2}\|x-x_{0}\|^{2}+\sum_{i=1}^{r}h_{i}(x),\label{eq:Dykstra-form-1}
\end{equation}
where $h_{i}(\cdot)$ are generalized to be closed convex functions,
as Dykstra's splitting. Instead of projections, one now uses proximal
mappings. (See \eqref{eq:primal-subpblm} for an example.) Dykstra's
splitting was studied in \cite{Han89} and \cite{Tseng-93-Dual-ascent}
for the case of $r\geq2$, and they proved the convergence (to the
primal minimizer) under constraint qualifications. It was also proved
in \cite{Bauschke-Combettes-Dykstra-split-2008} that Dykstra's splitting
converges for the case of $r=2$ without constraint qualifications. 

Dykstra's algorithm is related to the method of alternating projections
for finding a point in the intersection more than one closed set.
For more information on the various topics in Dykstra's algorithm
mentioned so far, we refer to \cite{Deutsch01_survey,Deustch01,BauschkeCombettes11,EsRa11}.

\subsection{Block coordinate minimization}

For the problem of minimizing $f(x)+g(x)$, where $f(\cdot)$ is smooth
and $g(\cdot)$ is block separable, one strategy is to minimize one
block of the variables at a time, keeping the others fixed. This strategy
is called block coordinate minimization, or alternating minimization.
Nonasymptotic convergence rates of $O(1/k)$ to the optimal value
were obtained for when the smooth function is not known to be strongly
convex in \cite{Beck_Tetruashvili_2013,Beck_alt_min_SIOPT_2015}.
We refer to these papers for more on the history of block coordinate
minimization. 

The smooth portion of the dual problem in Dykstra's algorithm is a
specific quadratic function, so block coordinate minimization for
this problem coincides with a block coordinate proximal gradient approach
in \cite{Tseng_YUn_MAPR_2009,Tseng_Yun_JOTA_2009}. Convergence properties
of minimizing over more than one block at a time were discussed. There
is too much recent research on block coordinate minimization and block
coordinate proximal gradient, so we refer the reader to the two recent
references \cite{Wright_BCD_MAPR2015,Hong_Wang_Razaviyayn_Luo_MP_BCD}
and their references within.

\subsection{Proximal point algorithm}

The proximal point algorithm attributed to \cite{Martinet,Rockafellar-Prox-Point}
is a method for finding minimizers of $\min_{x}f(x)$ by creating
a sequence $\{x_{j}\}_{j}$ such that 
\[
\begin{array}{c}
x_{j+1}\approx\mbox{prox}_{f}(x_{j}):=\arg\min_{x}f(x)+\frac{1}{2}\|x-x_{j}\|^{2}.\end{array}
\]
It was noticed in \cite{Han89} that one can use the proximal point
algorithm to solve \eqref{eq:main-pblm} by approximately solving
a sequence of problems of the form \eqref{eq:Dykstra-form-1} using
Dykstra's algorithm. The rules there for moving to a new proximal
center $x_{j}$ involves finding a primal feasible point that satisfies
the optimality conditions approximately. But such a feasible point
might not be found in a finite number of iterations when some of the
functions $h_{i}(\cdot)$ are indicator functions, so a separate rule
for moving the proximal center is needed.

\subsection{Other methods for minimizing the sum of functions}

When the constraint sets are either too big and have to be split up
as the intersection of more than 1 set, or when these constraint sets
are only revealed as the algorithm is run, it is beneficial to write
these problems in the form \eqref{eq:main-pblm} where two or more
of the $h_{i}(\cdot)$ are indicator functions. In such a case, as
remarked in \cite{Nedic_MP_2011}, the accelerated methods of \cite{Nesterov_book}
and further developed by \cite{BeckTeboulle2009,Tseng_APG_2008} do
not immediately apply (to the primal problem). We now recall other
methods and observations on minimizing \eqref{eq:main-pblm} when
more than one of the functions $h_{i}(\cdot)$ are indicator functions
and the algorithm can operate on a few of the functions $h_{i}(\cdot)$
at a time. As we have seen earlier, Dykstra's algorithm is one such
example. 

In the case where all the functions $h_{i}(\cdot)$ in \eqref{eq:main-pblm}
are indicator functions, then this problem coincides with the problem
of finding a point in the intersection of convex sets, which is a
problem of much interest on its own. (See for example \cite{EsRa11,BB96_survey,Deustch01}.)
We refer to this as the convex feasibility problem. The convex feasibility
problem can be solved by the method of alternating projections and
the Douglas-Rachford method. A discussion of the effectiveness of
methods for the convex feasibility problem is \cite{CensorChenCombettesDavidiHerman12}.

Beyond the convex feasibility problem, various extensions of the subgradient
method in \cite{Helou_Neto_De_Pierro_SIOPT_2009,Ram_Nedic_Veeravalli_SIOPT_2009,Nedic_MP_2011}
can solve problems of the form \eqref{eq:main-pblm}. Another recent
development is in superiorization (See for example \cite{Censor_Davidi_Herman_Superiorization}),
where an algorithm for the convex feasibility problem is perturbed
to try to reduce the value of the objective function. The result is
an algorithm that seeks feasibility at a rate comparable to algorithms
for the feasibility problem, while achieving a superior objective
value to what an algorithm for the feasibility problem alone would
achieve. A comparison of projected subgradient methods and superiorization
is given in \cite{Censor_Davidi_Herman_Schulte_Tetruashvili}. 

A typical assumption on the constraint sets is that they have a Lipschitzian
error bound, which is also equivalent to the stability of the intersection
under perturbations. See for example \cite{BBL99,Burke_Deng_2005,Ng_Yang_2004,Kruger_06}. 

Lastly, another method for minimizing \eqref{eq:main-pblm} is the
ADMM \cite{Boyd_Eckstein_ADMM_review}. The ADMM is an effective method,
but we feel that Dykstra's splitting still has its own value. For
example, as we shall see later, the different agents can minimize
in any order, and convergence doesn't even require the existence of
a dual minimizer. In problems where the different agents are assumed
not to be able to freely communicate between each other or if communications
between two agents are one dimensional, methods derived from subgradient
algorithms can still be a method of choice \cite{Nedich_survey},
even though many algorithms are preferred over the subgradient algorithm
in large scale problems with less restrictive communcation requirements
\cite{Nesterov_book}.

We refer to the survey \cite{Combettes_Pesquet} for other proximal
techniques for minimizing \eqref{eq:main-pblm}.

\subsection{\label{sub:Contrib}Contributions of this paper}

Firstly, in Section \ref{sec:Dykstra-splitting}, we extend Dykstra's
splitting for minimizing \eqref{eq:Dykstra-form-1} so that
\begin{itemize}
\item [(A)]the proof of convergence does not require constraint qualifications,
\item [(B)]the $r$ in \eqref{eq:Dykstra-form-1} is any number greater
than or equal to 2, and 
\item [(C)]$h_{i}(\cdot)$ can be any closed convex function instead of
the indicator function.
\end{itemize}
As mentioned earlier, \cite{BD86} and \cite{Gaffke_Mathar} have
features (A) and (B),  \cite{Han89} has (B) and (C), and \cite{Bauschke-Combettes-Dykstra-split-2008}
has (A) and (C). We are not aware of Dykstra's splitting being proved
to have features (A), (B) and (C). In addition, our analysis incorporates
these features that are now rather standard in block coordinate minimization
algorithms.
\begin{itemize}
\item [(D)]the convex functions $h_{i}(\cdot)$ are not necessarily sampled
in a cyclic manner like in \cite{Hundal-Deutsch-97},
\item [(E)]more than one convex function $h_{i}(\cdot)$ can be minimized
at one time in the Dykstra's splitting, and
\item [(F)]the SHQP strategy in \cite{Pang_DBAP} is applied.
\end{itemize}
The proof is largely adapted from \cite{Gaffke_Mathar}. This paper
also updates the discussion of the SHQP strategy in \cite{Pang_DBAP}
by pointing out that if the convex functions $\delta_{C_{i}}^{*}(\cdot)$
are not necessarily sampled in a cyclic manner, then we just need
one set of the form $\tilde{C}^{n,w}$ in Algorithm \ref{alg:Ext-Dyk}
instead of multiple sets of this type as was done in \cite{Pang_DBAP}. 

Secondly, in Section \ref{sec:approx-prox}, we show that one can
minimize problems of the form \eqref{eq:main-pblm} where the feasible
region is a compact set by combining Dykstra's splitting on problems
of the kind \eqref{eq:Dykstra-form-1} and an approximate proximal
point algorithm where the proximal center is moved once the KKT conditions
are approximately satisfied. The compactness of the feasible region
allows us to remove the constraint qualifications on the constraint
sets for our results.

In Section \ref{sec:O-1-n-convergence}, we show that if a dual minimizer
exists and some processing is performed so that the dual multipliers
related to the indicator functions are uniformly bounded throughout
all iterations, an $O(1/n)$ convergence of the dual problem (which
leads to an $O(1/\sqrt{n})$ convergence to the primal minimizer)
can be attained.

\subsection{Notation}

We use ``$\partial$'' to refer to either the subdifferential of
a convex function, or the boundary of a set, which should be clear
from context. The conjugate $\delta_{C}^{*}(\cdot)$ of the indicator
function has the form $\delta_{C}^{*}(y)=\sup_{x\in C}\langle y,x\rangle$,
and is also known as the support function.

\section{\label{sec:Dykstra-splitting}Dykstra splitting for the sum of convex
functions}

Consider the primal problem 
\begin{equation}
\begin{array}{c}
(P)\quad\alpha=\underset{x\in X}{\min}\frac{1}{2}\|x-x_{0}\|^{2}+\underset{i=1}{\overset{r_{1}}{\sum}}f_{i}(x)+\underset{i=r_{1}+1}{\overset{r_{2}}{\sum}}g_{i}(x)+\underset{i=r_{2}+1}{\overset{r}{\sum}}\delta_{C_{i}}(x),\end{array}\label{eq:primal}
\end{equation}
where $X$ is a finite dimensional Hilbert space, and 
\begin{itemize}
\item [(A1)]$f_{i}:X\to\mathbb{R}$ are convex functions such that $\mbox{dom}f_{i}(\cdot)=X$
for all $i\in\{1,\dots,r_{1}\}$.
\item [(A2)]$g_{i}:X\to\mathbb{R}$ are lower semicontinuous convex functions
for all $i\in\{r_{1}+1,\dots,r_{2}\}$.
\item [(A3)]$C_{i}$ are closed convex subsets of $X$ for all $i\in\{r_{2}+1,\dots,r\}$.
\end{itemize}
In this section, we generalize the proof in \cite{Gaffke_Mathar}
to show that Dykstra's splitting algorithm can be used to minimize
problems of the form \eqref{eq:primal}.

We note that the functions $\delta_{C_{i}}(\cdot)$ and $f_{i}(\cdot)$
can be written as $g_{i}(\cdot)$. But as we will see later, we will
treat the functions of the three types differently in Algorithm \ref{alg:Ext-Dyk}.
 For convenience of future discussions, let $h:X\to\mathbb{R}$ and
$h_{i}:X\to\mathbb{R}$ be the convex functions defined by 
\[
h(\cdot)=\sum_{i=1}^{r}h_{i}(\cdot)\mbox{, and }h_{i}(\cdot)=\begin{cases}
f_{i}(\cdot) & \mbox{ if }i\in\{1,\dots,r_{1}\}\\
g_{i}(\cdot) & \mbox{ if }i\in\{r_{1}+1,\dots,r_{2}\}\\
\delta_{C_{i}}(\cdot) & \mbox{ if }i\in\{r_{2}+1,\dots,r\},
\end{cases}
\]
so that the objective function in \eqref{eq:primal} can be written
simply as $\frac{1}{2}\|x-x_{0}\|^{2}+h(x)$.

\subsection{Algorithm description and commentary}

The (Fenchel) dual of problem \eqref{eq:primal} is 
\begin{equation}
(D)\quad\beta=\max_{z\in X^{r}}\,\,F(z),\label{eq:dual}
\end{equation}
where $F:X^{r}\to\mathbb{R}$ is defined by 
\begin{equation}
\begin{array}{c}
F(z)=-\frac{1}{2}\left\Vert \left(\underset{i=1}{\overset{r}{\sum}}z_{i}\right)-x_{0}\right\Vert ^{2}-\underset{i=1}{\overset{r}{\sum}}h_{i}^{*}(z_{i})+\frac{1}{2}\|x_{0}\|^{2}.\end{array}\label{eq:formula-F}
\end{equation}
By weak duality, we have $\beta\leq\alpha$. (Actually $\beta=\alpha$
is true; We will see that later.) 

If $\tilde{C}$ is any closed convex set such that $\bar{C}\subset\tilde{C}$,
where the set $\bar{C}$ is defined by 
\begin{equation}
\bar{C}:=[\cap_{i=r_{2}+1}^{r}C_{i}]\cap[\cap_{i=r_{1}+1}^{r_{2}}\mbox{cl dom}\,g_{i}(\cdot)],\label{eq:def-C-bar}
\end{equation}
then problem \eqref{eq:primal} has the same (primal) minimizer as
\begin{equation}
\begin{array}{c}
(P_{\tilde{C}})\quad\alpha=\underset{x\in X}{\min}\frac{1}{2}\|x-x_{0}\|^{2}+\underset{i=1}{\overset{r}{\sum}}h_{i}(x)+\delta_{\tilde{C}}(x).\end{array}\label{eq:SHQP-P}
\end{equation}
The dual of $(P_{\tilde{C}})$ is 
\[
(D_{\tilde{C}})\quad\beta=\max_{z\in X^{r+1}}\,\,F_{\tilde{C}}(z),
\]
where $F_{\tilde{C}}:X^{r+1}\to\mathbb{R}$ is defined by 
\begin{equation}
\begin{array}{c}
F_{\tilde{C}}(z)=-\frac{1}{2}\left\Vert \left(\underset{i=1}{\overset{r+1}{\sum}}z_{i}\right)-x_{0}\right\Vert ^{2}-\underset{i=1}{\overset{r}{\sum}}h_{i}^{*}(z_{i})-\delta_{\tilde{C}}^{*}(z_{r+1})+\frac{1}{2}\|x_{0}\|^{2}.\end{array}\label{eq:SHQP-dual}
\end{equation}

As detailed in \cite{Pang_DBAP}, this observation leads us to construct
a set $\tilde{C}^{n,w}$ that changes in each iteration of our extended
Dykstra's algorithm in Algorithm \ref{alg:Ext-Dyk} below.

\begin{algorithm}
\begin{lyxalgorithm}
\label{alg:Ext-Dyk}(Extended Dykstra's algorithm) Consider the problem
\eqref{eq:primal} along with the associated problems \eqref{eq:dual},
\eqref{eq:SHQP-P} and \eqref{eq:SHQP-dual}. 

Set some number $M\in\mathbb{R}_{+}\cup\{\infty\}$, and let $\bar{w}$
be a positive integer. Our extended Dykstra's algorithm is as follows:

01$\quad$Define the set $H^{1,0}$ to be $H^{1,0}=X$. 

02$\quad$Let $z^{1,0}\in X^{r+1}$ be the starting dual vector for
\eqref{eq:SHQP-dual}, and let $z_{r+1}^{1,0}=0$. 

03$\quad$Let $x^{1,0}$ be $x^{1,0}=x_{0}-\sum_{i=1}^{r+1}z_{i}^{1,0}$.

04$\quad$For $n=1,2,\dots$

05$\quad$$\quad$For $w=1,2,\dots,\bar{w}$

06$\quad$$\quad$$\quad$Choose a subset $S_{n,w}\subset\{1,\dots,r+1\}$.

07$\quad$$\quad$$\quad$If $r+1\in S_{n,w}$, then 

$\phantom{\mbox{07}}$$\quad$$\quad$$\quad$$\quad$\textbf{Dual
decrease with SHQP steps}

08$\quad$$\quad$$\quad$$\quad$Choose $\tilde{C}^{n,w}$ to be
any set such that $\bar{C}\subset\tilde{C}^{n,w}\subset H^{n,w-1}$. 

09$\quad$$\quad$$\quad$$\quad$Let $z_{i}^{n,w}=z_{i}^{n,w-1}$
for all $i\notin S_{n,w}$

10$\quad$$\quad$$\quad$$\quad$Let $\{z_{i}^{n,w}\}_{i\in S_{n,w}}$
be defined through 
\begin{eqnarray}
z^{n,w}= & \underset{z\in X^{r+1}}{\arg\max} & -\frac{1}{2}\left\Vert \left(\sum_{i\in S_{n,w}}z_{i}+\sum_{i\notin S_{n,w}}z_{i}^{n,w-1}\right)-x_{0}\right\Vert ^{2}\nonumber \\
 &  & \qquad-\sum_{i=1}^{r}h_{i}^{*}(z_{i})-\delta_{\tilde{C}^{n,w}}^{*}(z_{r+1})+\frac{1}{2}\|x_{0}\|^{2}.\nonumber \\
 & \mbox{s.t.} & z_{i}^{n,w}=z_{i}^{n,w-1}\mbox{ for all }i\notin S_{n,w}.\label{eq:dual-obj-fn}
\end{eqnarray}

11$\quad$$\quad$$\quad$$\quad$Let $H^{n,w}$ be a set such that
$\delta_{\tilde{C}^{n,w}}^{*}(z_{r+1}^{n,w})=\delta_{H^{n,w}}^{*}(z_{r+1}^{n,w})$
and $\bar{C}\subset H^{n,w}$. 

12$\quad$$\quad$$\quad$Else

$\phantom{\mbox{07}}$$\quad$$\quad$$\quad$$\quad$\textbf{Dual
decrease }

13$\quad$$\quad$$\quad$$\quad$Let $z_{i}^{n,w}=z_{i}^{n,w-1}$
for all $i\notin S_{n,w}$. 

14$\quad$$\quad$$\quad$$\quad$Define $\{z_{i}^{n,w}\}_{i\in S_{n,w}}$
through \eqref{eq:dual-obj-fn}, except with $\delta_{\tilde{C}^{n,w}}^{*}(z_{r+1})$
omitted

15$\quad$$\quad$$\quad$$\quad$Let $H^{n,w}$ be $H^{n,w-1}$. 

16$\quad$$\quad$$\quad$End If 

17$\quad$$\quad$End For 

$\phantom{\mbox{07}}$$\quad$$\quad$\textbf{Aggregating variables}

18$\quad$$\quad$Find $z^{n+1,0}\in X^{r+1}$ and $H^{n+1,0}\supset\bar{C}$
such that \begin{subequations}\label{eq_m:aggregate} 
\begin{eqnarray}
z_{i}^{n+1,0} & = & z_{i}^{n,\bar{w}}\mbox{ for all }i\in\{1,\dots,r_{2}\}\label{eq:aggregate-1}\\
\sum_{i=1}^{r+1}z_{i}^{n+1,0} & = & \sum_{i=1}^{r+1}z_{i}^{n,\bar{w}}\label{eq:aggregate-2}\\
\|z_{i}^{n+1,0}\| & \leq & M\mbox{ for all }i\in\{r_{2}+1,\dots,r\}\label{eq:aggregate-3}\\
\!\!\!\!\sum_{i=r_{2}+1}^{r}\delta_{C_{i}}^{*}(z_{i}^{n+1,0})+\delta_{H^{n+1,0}}^{*}(z_{r+1}^{n+1,0}) & \leq & \sum_{i=r_{2}+1}^{r}\delta_{C_{i}}^{*}(z_{i}^{n,\bar{w}})+\delta_{H^{n,\bar{w}}}^{*}(z_{r+1}^{n,\bar{w}})\label{eq:aggregate-4}\\
\sum_{i=1}^{r+1}\|z_{i}^{n+1,0}\| & \leq & \sum_{i=1}^{r+1}\|z_{i}^{n,\bar{w}}\|.\label{eq:aggregate-5}
\end{eqnarray}
\end{subequations}

19$\quad$End For \end{lyxalgorithm}
\end{algorithm}

We list some observations of Algorithm \ref{alg:Ext-Dyk}. The choice
of $S_{n,w}$ in line 6 of Algorithm \ref{alg:Ext-Dyk} allows for
more than one block of $z$ to be minimized in \eqref{eq:dual-obj-fn}.
If $\bar{w}=r+1$, the sets $S_{n,w}$ are chosen to be $\{w\}$,
and $r_{1}=r_{2}=0$, then Algorithm \ref{alg:Ext-Dyk} reduces to
the extended Dykstra's algorithm that was discussed in \cite{Pang_DBAP}. 
\begin{rem}
(Choice of $H^{n,w}$) An easy choice for $H^{n,w}$ in line 11 of
Algorithm \ref{alg:Ext-Dyk} is to choose a halfspace with outward
normal $z_{r+1}^{n,w}$ that supports the set $\tilde{C}^{n,w}$.
Another example of $H^{n,w}$ is the intersection of the halfspace
mentioned earlier with a small number of halfspaces containing $\bar{C}$
defined in \eqref{eq:def-C-bar} that will allow $H^{n,w}$ to approximate
$\bar{C}$ well.
\end{rem}
We have the following identities to simplify notation: \begin{subequations}\label{eq_m:from-10-13}
\begin{eqnarray}
v^{n,w} & := & \begin{array}{c}
\underset{j=1}{\overset{r+1}{\sum}}z_{j}^{n,w}\end{array}\label{eq:from-10}\\
\mbox{ and }x^{n,w} & := & \begin{array}{c}
x_{0}-v^{n,w}.\end{array}\label{eq:From-13}
\end{eqnarray}
\end{subequations}
\begin{claim}
\label{claim:Fenchel-duality}For all $i\in S_{n,w}$, we have 
\begin{enumerate}
\item [(a)]$-x^{n,w}+\partial h_{i}^{*}(z_{i}^{n,w})\ni0$,
\item [(b)]$-z_{i}^{n,w}+\partial h_{i}(x^{n,w})\ni0$, and
\item [(c)]$h_{i}(x^{n,w})+h_{i}^{*}(z_{i}^{n,w})=\langle x^{n,w},z_{i}^{n,w}\rangle$. 
\end{enumerate}
\end{claim}
\begin{proof}
By taking the optimality conditions in \eqref{eq:dual-obj-fn} with
respect to $z_{i}$ for $i\in S_{n,w}$, we deduce (a). The equivalences
of (a), (b) and (c) is standard. 
\end{proof}
Dykstra's algorithm is traditionally written in terms of solving for
the primal variable $x$. For completeness, we show the equivalence
between \eqref{eq:dual-obj-fn} and the primal minimization problem. 
\begin{prop}
(On solving \eqref{eq:dual-obj-fn}) If a minimizer $z^{n,w}$ for
\eqref{eq:dual-obj-fn} exists, then the $x^{n,w}$ in \eqref{eq:From-13}
satisfies 
\begin{equation}
x^{n,w}=\begin{array}{c}
\underset{x\in X}{\arg\min}\underset{i\in S_{n,w}}{\sum}h_{i}(x)+\frac{1}{2}\left\Vert x-\left(x_{0}-\underset{i\notin S_{n,w}}{\sum}z_{i}^{n,w}\right)\right\Vert ^{2}.\end{array}\label{eq:primal-subpblm}
\end{equation}
Conversely, if $x^{n,w}$ solves \eqref{eq:primal-subpblm} with the
dual variables $\{\tilde{z}_{i}^{n,w}\}_{i\in S_{n,w}}$ satisfying
\begin{equation}
\begin{array}{c}
\tilde{z}_{i}^{n,w}\in\partial h_{i}(x^{n,w})\mbox{ and }x^{n,w}-x_{0}+\underset{i\notin S_{n,w}}{\sum}z_{i}^{n,w}+\underset{i\in S_{n,w}}{\sum}\tilde{z}_{i}^{n,w}=0,\end{array}\label{eq:primal-optim-cond}
\end{equation}
then $\{\tilde{z}_{i}^{n,w}\}_{i\in S_{n,w}}$ solves \eqref{eq:dual-obj-fn}. \end{prop}
\begin{proof}
For the first part, note that 
\begin{eqnarray*}
 &  & \begin{array}{c}
\partial\left(h+\frac{1}{2}\|\cdot-(x_{0}-\underset{i\notin S_{n,w}}{\sum}z_{i}^{n,w})\|^{2}\right)(x^{n,w})\end{array}\\
 & \supset & \begin{array}{c}
\underset{i\in S_{n,w}}{\sum}\partial h_{i}(x^{n,w})+[x^{n,w}-(x_{0}-\underset{i\notin S_{n,w}}{\sum}z_{i}^{n,w})]\end{array}\\
 & \overset{\scriptsize\mbox{Claim \ref{claim:Fenchel-duality}(b)}}{\ni} & \begin{array}{c}
\underset{i\in S_{n,w}}{\sum}z_{i}^{n,w}+x^{n,w}-x_{0}+\underset{i\notin S_{n,w}}{\sum}z_{i}^{n,w}\overset{\eqref{eq_m:from-10-13}}{=}0.\end{array}
\end{eqnarray*}
For the second part, note that the first part of \eqref{eq:primal-optim-cond}
implies that $x^{n,w}\in\partial h_{i}^{*}(\tilde{z}_{i}^{n,w})$,
while the second part of \eqref{eq:primal-optim-cond} implies that
$0$ lies in the subdifferential of the objective function in \eqref{eq:dual-obj-fn}. \end{proof}
\begin{rem}
(Information needed to calculate \eqref{eq:dual-obj-fn}) We note
that in \eqref{eq:dual-obj-fn}, one only needs to have knowledge
of the variables $v^{n,w-1}$ and $z_{i}^{n,w-1}$ for $i\in S_{n,w}$.
Thus Dykstra's splitting may be suitable for problems where the communication
costs is high compared to the costs of solving the proximal problems. 
\begin{rem}
(On line 18 of Algorithm \ref{alg:Ext-Dyk}) If $M=\infty$ in Algorithm
\ref{alg:Ext-Dyk}, then $z^{n+1,0}$ and $H^{n+1,0}$ can be set
to be $z^{n,\bar{w}}$ and $H^{n,\bar{w}}$ respectively. We had to
add this line to Algorithm \ref{alg:Ext-Dyk} because the boundedness
condition \eqref{eq:aggregate-3} is necessary for our $O(1/n)$ convergence
result in Section \ref{sec:O-1-n-convergence}. This detail can be
skipped for the discussions in this section and Section \ref{sec:approx-prox}.
\end{rem}
\end{rem}
We need the following fact before we discuss how to find $z^{n+1,0}$
and $H^{n+1,0}$ satisfying \eqref{eq_m:aggregate}.
\begin{fact}
\label{fact:agg-halfspaces}(Aggregating halfspaces) Consider two
halfspaces, say $H_{1}$ and $H_{2}$, which have (outward) normals
$z_{1}$ and $z_{2}$. Assume that $\{z_{1},z_{2}\}$ are linearly
independent. Construct a third halfspace $H_{3}$ with normal $z_{1}+z_{2}$
such that $H_{3}\supset H_{1}\cap H_{2}$ and $\partial H_{3}\cap[H_{1}\cap H_{2}]\neq\emptyset$.
Let $x$ be any point on $\partial H_{1}\cap\partial H_{2}$. We see
that $x\in\partial H_{3}$. We have 
\[
\delta_{H_{1}}^{*}(z_{1})+\delta_{H_{2}}^{*}(z_{2})=\langle z_{1},x\rangle+\langle z_{2},x\rangle=\langle z_{1}+z_{2},x\rangle=\delta_{H_{3}}^{*}(z_{1}+z_{2}).
\]
If $\{z_{1},z_{2}\}$ is linearly dependent instead, then $H_{3}:=H_{1}\cap H_{2}$
is a halfspace, and $\delta_{H_{3}}^{*}(z_{1}+z_{2})\leq\delta_{H_{1}}^{*}(z_{1})+\delta_{H_{2}}^{*}(z_{2})$.
Moreover, the inequality is strict if, for example, $z_{1}\neq0$,
$z_{2}\neq0$ and $H_{1}\subsetneq H_{2}$. This fact can be generalized
for more than two halfspaces.
\end{fact}
We state some notation necessary for further discussions. For any
$i\in\{1,\dots,r+1\}$ and $n\in\{1,2,\dots\}$, let $p(n,i)$ be
\[
p(n,i)=\max\{m:m\leq\bar{w},i\in S_{n,m}\}.
\]
In other words, $p(n,i)$ is the index $m$ such that $i\in S_{n,m}$
but $i\notin S_{n,k}$ for all $k\in\{m+1,\dots,\bar{w}\}$. It follows
from lines 9 and 13 of Algorithm \ref{alg:Ext-Dyk} that 
\begin{equation}
z_{i}^{n,p(n,i)}=z_{i}^{n,p(n,i)+1}=\cdots=z_{i}^{n,\bar{w}}.\label{eq:stagnant-indices}
\end{equation}
We now show one way to find $z^{n+1,0}$ and $H^{n+1,0}$ satisfying
\eqref{eq:aggregate-4}.
\begin{prop}
\label{prop:aggregate-aim}(On satisfying \eqref{eq:aggregate-4})
Set $z_{i}^{n+1,0}=\alpha_{i}z_{i}^{n,\bar{w}}$ for all $i\in\{r_{2}+1,\dots,r\}$,
where $\alpha_{i}$ is a number in $[0,1]$, so that \eqref{eq:aggregate-3}
is satisfied. Then 
\begin{equation}
z_{r+1}^{n+1,0}\overset{\eqref{eq:aggregate-1},\eqref{eq:aggregate-2}}{=}\sum_{i=r_{2}+1}^{r+1}z_{i}^{n,\bar{w}}-\sum_{i=r_{2}+1}^{r}z_{i}^{n+1,0}=z_{r+1}^{n,\bar{w}}+\sum_{i=r_{2}+1}^{r}(1-\alpha_{i})z_{i}^{n,\bar{w}}.\label{eq:to-get-aggregate-5}
\end{equation}
For $i\in\{r_{2}+1,\dots,r\}$, recall that by the construction of
$z_{i}^{n,p(n,i)}$ in \eqref{eq:dual-obj-fn} and \eqref{eq:primal-subpblm},
the condition \eqref{eq:primal-optim-cond} implies that $\tilde{H}^{n,i}\supset C_{i}$,
where the halfspace $\tilde{H}^{n,i}$ is defined by 
\begin{eqnarray}
\tilde{H}^{n,i} & := & \{x:\langle x-x^{n,p(n,i)},z_{i}^{n,p(n,i)}\rangle\leq0\}\label{eq:derived-halfspace}\\
 & \overset{\eqref{eq:stagnant-indices}}{=} & \{x:\langle x-x^{n,p(n,i)},z_{i}^{n,\bar{w}}\rangle\leq0\}.\nonumber 
\end{eqnarray}
and $x^{n,w}$ is as defined in \eqref{eq:primal-subpblm}. We can
check that 
\[
\delta_{C_{i}}^{*}(\alpha z_{i}^{n,\bar{w}})=\delta_{\tilde{H}^{n,i}}^{*}(\alpha z_{i}^{n,\bar{w}})\mbox{ for any }\alpha\geq0\mbox{ and }i\in\{r_{2}+1,\dots,r\}.
\]
Let $I_{n}\subset\{r_{2}+1,\dots,r\}$ be the set of indices $i$
such that $z_{i}^{n+1,0}\neq z_{i}^{n,\bar{w}}$. Let $H^{n+1,0}$
be the halfspace with outward normal $z_{r+1}^{n+1,0}$ such that
\begin{eqnarray*}
 &  & \begin{array}{c}
H^{n+1,0}\supset H^{n,\bar{w}}\cap\bigcap_{i\in I_{n}}\tilde{H}^{n,i}\end{array}\\
 & \mbox{ and } & \begin{array}{c}
\partial H^{n+1,0}\cap[H^{n,\bar{w}}\cap\bigcap_{i\in I_{n}}\tilde{H}^{n,i}]\neq\emptyset.\end{array}
\end{eqnarray*}
Then \eqref{eq:aggregate-4} is satisfied. Furthermore, \eqref{eq:aggregate-4}
is actually an equality if the normals $\{z_{i}^{n,\bar{w}}\}_{i\in I_{n}\cup\{r+1\}}$
are linearly independent. \end{prop}
\begin{proof}
The conclusion can be deduced from Fact \ref{fact:agg-halfspaces}.
\end{proof}
One can check that the construction in Proposition \ref{prop:aggregate-aim}
also leads to the conditions in \eqref{eq_m:aggregate}. In particular,
\eqref{eq:aggregate-5} can be inferred from \eqref{eq:to-get-aggregate-5}
via 
\begin{eqnarray*}
\sum_{i=1}^{r+1}\|z_{i}^{n+1,0}\| & = & \|z_{r+1}^{n+1,0}\|+\sum_{i=1}^{r_{2}}\|z_{i}^{n+1,0}\|+\sum_{i=r_{2}+1}^{r}\|z_{i}^{n+1,0}\|\\
 & \overset{\eqref{eq:to-get-aggregate-5}}{\leq} & \|z_{r+1}^{n,\bar{w}}\|+\sum_{i=1}^{r_{2}}\|z_{i}^{n,\bar{w}}\|+\sum_{i=r_{2}+1}^{r}(\alpha_{i}+(1-\alpha_{i}))\|z_{i}^{n,\bar{w}}\|\\
 & = & \sum_{i=1}^{r+1}\|z_{i}^{n,\bar{w}}\|.
\end{eqnarray*}
 The other items in \eqref{eq_m:aggregate} are clear.

\subsection{Convergence of Algorithm \ref{alg:Ext-Dyk}}

We now prove the convergence of Algorithm \ref{alg:Ext-Dyk}. We first
list assumptions that will ensure convergence to the primal minimizer. 
\begin{assumption}
\label{assu:Assumptions}We make a few assumptions on Algorithm \ref{alg:Ext-Dyk}:
\begin{enumerate}
\item [(a)]The objective value $\alpha$ in \eqref{eq:primal} is a finite
number.
\item [(b)]The sets $S_{n,w}\subset\{1,\dots,r+1\}$ are chosen such that
for all $n$, $\cup_{w=1}^{\bar{w}}S_{n,w}=\{1,\dots,r+1\}$.
\item [(c)]There are constants $A$ and $B$ such that for all $n$, $\sum_{i=1}^{r+1}\|z_{i}^{n,w}\|\leq\sqrt{n}A+B$.
\item [(d)]Minimizers of \eqref{eq:dual-obj-fn} can be obtained in each
step.
\end{enumerate}
\end{assumption}
We give a brief commentary on Assumption \ref{assu:Assumptions}.
Assumption \ref{assu:Assumptions}(a) together with the strong convexity
of the primal problem says that \eqref{eq:primal} is feasible and
a unique primal minimizer exists. As we will see later, the structure
of the functions $f_{i}(\cdot)$ for $i\in\{1,\dots,r_{1}\}$ implies
that $z_{i}^{n,w}$ is uniformly bounded for all $i\in\{1,\dots,r_{1}\}$.
In Proposition \ref{prop:satisfy-sqrt-growth}, we shall introduce
a condition on the choice of $S_{n,w}$ that will ensure that Assumption
\ref{assu:Assumptions}(c) is satisfied. 

 We follow the proof in \cite{Gaffke_Mathar} to show that $\lim_{n\to\infty}x^{n,\bar{w}}$
exists and is the minimizer of (P). 

For any $x\in X$ and $z\in X^{r+1}$, the analogue of \cite[(8)]{Gaffke_Mathar}
is 
\begin{eqnarray}
 &  & \frac{1}{2}\|x_{0}-x\|^{2}+\sum_{i=1}^{r}h_{i}(x)+\delta_{\tilde{C}}(x)-F_{\tilde{C}}(z_{1},\dots,z_{r},z_{r+1})\label{eq:From-8}\\
 & \overset{\eqref{eq:SHQP-dual}}{=} & \frac{1}{2}\|x_{0}-x\|^{2}+\sum_{i=1}^{r}[h_{i}(x)+h_{i}^{*}(z_{i})]-\left\langle x_{0},\sum_{i=1}^{r+1}z_{i}\right\rangle +\frac{1}{2}\left\Vert \sum_{i=1}^{r+1}z_{i}\right\Vert ^{2}\nonumber \\
 &  & +\delta_{\tilde{C}}(x)+\delta_{\tilde{C}}^{*}(z_{r+1})\nonumber \\
 & \overset{\scriptsize\mbox{Fenchel duality}}{\geq} & \frac{1}{2}\|x_{0}-x\|^{2}+\sum_{i=1}^{r+1}\langle x,z_{i}\rangle-\left\langle x_{0},\sum_{i=1}^{r+1}z_{i}\right\rangle +\frac{1}{2}\left\Vert \sum_{i=1}^{r+1}z_{i}\right\Vert ^{2}\nonumber \\
 & = & \frac{1}{2}\left\Vert x_{0}-x-\sum_{i=1}^{r+1}z_{i}\right\Vert ^{2}\geq0.\nonumber 
\end{eqnarray}

The theorem below generalizes \cite[Theorem 1]{Gaffke_Mathar} for
the setting \eqref{eq:primal}.
\begin{thm}
\label{thm:convergence}Suppose Assumption \ref{assu:Assumptions}
holds. For the sequence $\{z^{n,w}\}_{{1\leq n<\infty\atop 0\leq w\leq\bar{w}}}\subset X^{r+1}$
generated by Algorithm \ref{alg:Ext-Dyk} and the sequences $\{v^{n,w}\}_{{1\leq n<\infty\atop 0\leq w\leq\bar{w}}}\subset X$
and $\{x^{n,w}\}_{{1\leq n<\infty\atop 0\leq w\leq\bar{w}}}\subset X$
deduced from \eqref{eq_m:from-10-13}, we have:
\begin{enumerate}
\item [(i)]The sum $\sum_{n=1}^{\infty}\sum_{w=1}^{\bar{w}}\|v^{n,w}-v^{n,w-1}\|^{2}$
is finite and $\{F_{H^{n,\bar{w}}}(z^{n,\bar{w}})\}_{n=1}^{\infty}$
is nondecreasing.
\item [(ii)]There is a constant $C$ such that $\|v^{n,w}\|^{2}\leq C$
for all $n\in\mathbb{N}$ and $w\in\{1,\dots,\bar{w}\}$. 
\item [(iii)]There exists a subsequence $\{v^{n_{k},\bar{w}}\}_{k=1}^{\infty}$
of $\{v^{n,\bar{w}}\}_{n=1}^{\infty}$ which converges to some $v^{*}\in X$
and that 
\[
\lim_{k\to\infty}\langle v^{n_{k},\bar{w}}-v^{n_{k},p(n_{k},i)},z_{i}^{n_{k},\bar{w}}\rangle=0\mbox{ for all }i\in\{1,\dots,r+1\}.
\]

\item [(iv)]For the $v^{*}$ in (iii), $x_{0}-v^{*}$ is the minimizer
of the primal problem (P) and $\lim_{k\to\infty}F_{H^{n_{k},\bar{w}}}(z^{n_{k},\bar{w}})=\frac{1}{2}\|v^{*}\|^{2}+\sum_{i=1}^{r}h_{i}(x_{0}-v^{*})$. 
\end{enumerate}
The properties (i) to (iv) in turn imply that $\lim_{n\to\infty}x^{n,\bar{w}}$
exists, and $x_{0}-v^{*}$ is the primal minimizer of \eqref{eq:primal}.\end{thm}
\begin{proof}
 We first show that (i) to (iv) implies the final assertion. For
all $n\in\mathbb{N}$ we have, from weak duality, 
\begin{equation}
\begin{array}{c}
F_{H^{n,\bar{w}}}(z^{n,\bar{w}})\leq\beta\leq\alpha\leq\frac{1}{2}\|x_{0}-(x_{0}-v^{*})\|^{2}+\underset{i=1}{\overset{r}{\sum}}h_{i}(x_{0}-v^{*}),\end{array}\label{eq:weak-duality}
\end{equation}
hence $\beta=\alpha=\frac{1}{2}\|x_{0}-(x_{0}-v^{*})\|^{2}+h(x_{0}-v^{*})$,
and that $x_{0}-v^{*}=\arg\min_{x}h(x)+\frac{1}{2}\|x-x_{0}\|^{2}$.
Since the values $\{F_{H^{n,\bar{w}}}(z^{n,\bar{w}})\}_{n=1}^{\infty}$
are nondecreasing in $n$, we have 
\[
\begin{array}{c}
\underset{n\to\infty}{\lim}F_{H^{n,\bar{w}}}(z^{n,\bar{w}})=\frac{1}{2}\|x_{0}-(x_{0}-v^{*})\|^{2}+\underset{i=1}{\overset{r}{\sum}}h_{i}(x_{0}-v^{*}),\end{array}
\]
and (substituting $x=x_{0}-v^{*}$ in \eqref{eq:From-8}) 
\begin{eqnarray*}
 &  & \begin{array}{c}
\frac{1}{2}\|x_{0}-(x_{0}-v^{*})\|^{2}+h(x_{0}-v^{*})-F_{H^{n,\bar{w}}}(z^{n,\bar{w}})\end{array}\\
 & \overset{\eqref{eq:From-8},\eqref{eq:from-10}}{\geq} & \begin{array}{c}
\frac{1}{2}\|x_{0}-(x_{0}-v^{*})-v^{n,\bar{w}}\|^{2}\end{array}\\
 & \overset{\eqref{eq:From-13}}{=} & \begin{array}{c}
\frac{1}{2}\|x^{n,\bar{w}}-(x_{0}-v^{*})\|^{2}.\end{array}
\end{eqnarray*}
Hence $\lim_{n\to\infty}x^{n,\bar{w}}$ is the minimizer in (P). 

It remains to prove assertions (i) to (iv).

\textbf{Proof of (i):} We note that if $r+1\in S_{n,w}$, then 
\begin{eqnarray}
F_{H^{n,w-1}}(z^{n,w-1}) & \overset{\tilde{C}^{n,w}\subset H^{n,w-1}}{\leq} & \begin{array}{c}
F_{\tilde{C}^{n,w}}(z^{n,w-1})\end{array}\label{eq:SHQP-decrease}\\
 & \overset{\eqref{eq:dual-obj-fn},\eqref{eq:from-10}}{\leq} & \begin{array}{c}
F_{\tilde{C}^{n,w}}(z^{n,w})-\frac{1}{2}\|v^{n,w}-v^{n,w-1}\|^{2}\end{array}\nonumber \\
 & \overset{\scriptsize\mbox{Alg \ref{alg:Ext-Dyk}, line 11}}{=} & \begin{array}{c}
F_{H^{n,w}}(z^{n,w})-\frac{1}{2}\|v^{n,w}-v^{n,w-1}\|^{2}.\end{array}\nonumber 
\end{eqnarray}

The first inequality comes from the fact that since $\tilde{C}^{n,w}\subset H^{n,w-1}$
(from line 8 of Algorithm \ref{alg:Ext-Dyk}), then $\delta_{\tilde{C}^{n,w}}^{*}(\cdot)\leq\delta_{H^{n,w-1}}^{*}(\cdot)$.
The second inequality comes from the fact that $\{z_{i}^{n,w}\}_{i\in S_{n,w}}$
is a minimizer of the mapping 
\[
\begin{array}{c}
\{z_{i}\}_{i\in S_{n,w}}\mapsto\underset{i\in S_{n,w}}{\sum}h_{i}^{*}(z_{i})+\frac{1}{2}\left\Vert \left(\underset{i\in S_{n,w}}{\sum}z_{i}\right)-\left(x_{0}-\underset{i\notin S_{n,w}}{\sum}z_{i}^{n,w}\right)\right\Vert ^{2},\end{array}
\]
with the $\frac{1}{2}\|v^{n,w}-v^{n,w-1}\|^{2}$ arising from the
quadratic term.

When $r+1\notin S_{n,w}$, then we can make use of the fact that $z_{r+1}^{n,w}=z_{r+1}^{n,w-1}$
and $\tilde{C}^{n,w}=H^{n,w-1}=H^{n,w}$ to see that the inequality
\eqref{eq:SHQP-decrease} carries through as well. 

Recall that through \eqref{eq:aggregate-4}, $F_{H^{n+1,0}}(z^{n+1,0})\geq F_{H^{n,\bar{w}}}(z^{n,\bar{w}})$.
Combining \eqref{eq:SHQP-decrease} over all $m\in\{1,\dots,n\}$
and $w\in\{1,\dots,\bar{w}\}$, we have 
\[
\begin{array}{c}
F_{H^{1,0}}(z^{1,0})+\underset{m=1}{\overset{n}{\sum}}\underset{w=1}{\overset{\bar{w}}{\sum}}\|v^{m,w}-v^{m,w-1}\|^{2}\leq F_{H^{n,\bar{w}}}(z^{n,\bar{w}}).\end{array}
\]
Next, $F_{H^{n,\bar{w}}}(z^{n,\bar{w}})\leq\alpha$ by weak duality.
The proof of the claim is complete.

\textbf{Proof of (ii):} Substituting $x$ in \eqref{eq:From-8} to
be the primal minimizer $x^{*}$ and $z$ to be $z^{n,w}$, we have
\begin{eqnarray*}
 &  & \begin{array}{c}
\frac{1}{2}\|x_{0}-x^{*}\|^{2}+\underset{i=1}{\overset{r}{\sum}}h_{i}(x^{*})-F_{H^{1,0}}(z^{1,0})\end{array}\\
 & \overset{\scriptsize\mbox{part (i)}}{\geq} & \begin{array}{c}
\frac{1}{2}\|x_{0}-x^{*}\|^{2}+\underset{i=1}{\overset{r}{\sum}}h_{i}(x^{*})-F_{H^{n,w}}(z^{n,w})\end{array}\\
 & \overset{\eqref{eq:From-8}}{\geq} & \begin{array}{c}
\frac{1}{2}\left\Vert x_{0}-x^{*}-\underset{i=1}{\overset{r+1}{\sum}}z_{i}^{n,w}\right\Vert ^{2}\overset{\eqref{eq:from-10}}{=}\frac{1}{2}\|x_{0}-x^{*}-v^{n,w}\|^{2}.\end{array}
\end{eqnarray*}
The conclusion is immediate.

\textbf{Proof of (iii): } We first make use of the technique in
\cite[Lemma 29,1]{BauschkeCombettes11} (which is in turn largely
attributed to \cite{BD86}) to show that 
\begin{equation}
\begin{array}{c}
\underset{n\to\infty}{\liminf}\left[\left(\underset{w=1}{\overset{\bar{w}}{\sum}}\|v^{n,w}-v^{n,w-1}\|\right)\sqrt{n}\right]=0.\end{array}\label{eq:root-n-dec}
\end{equation}
Seeking a contradiction, suppose instead that there is an $\epsilon>0$
and $\bar{n}>0$ such that if $n>\bar{n}$, then $\left(\sum_{w=1}^{\bar{w}}\|v^{n,w}-v^{n,w-1}\|\right)\sqrt{n}>\epsilon$.
By the Cauchy Schwarz inequality, we have $\begin{array}{c}
\frac{\epsilon^{2}}{n}<\left(\underset{w=1}{\overset{\bar{w}}{\sum}}\|v^{n,w}-v^{n,w-1}\|\right)^{2}\leq\bar{w}\underset{w=1}{\overset{\bar{w}}{\sum}}\|v^{n,w}-v^{n,w-1}\|^{2}.\end{array}$ This contradicts the earlier claim in (i) that $\sum_{n=1}^{\infty}\sum_{w=1}^{\bar{w}}\|v^{n,w}-v^{n,w-1}\|^{2}$
is finite. 

Next, we recall Assumption \ref{assu:Assumptions}(c) that there are
constants $A$ and $B$ such that $\sum_{i=1}^{r+1}\|z_{i}^{n,\bar{w}}\|\leq A\sqrt{n}+B$
for all $n$. Through \eqref{eq:root-n-dec}, we find a sequence $\{n_{k}\}_{k=1}^{\infty}$
such that $\lim_{k\to\infty}\left[\left(\sum_{w=1}^{\bar{w}}\|v^{n_{k},w}-v^{n_{k},w-1}\|\right)\sqrt{n_{k}}\right]=0$.
Thus 
\begin{equation}
\begin{array}{c}
\underset{k\to\infty}{\lim}\left[\left(\underset{w=1}{\overset{\bar{w}}{\sum}}\|v^{n_{k},w}-v^{n_{k},w-1}\|\right)\|z_{i}^{n_{k},\bar{w}}\|\right]=0\mbox{ for all }i\in\{1,\dots,r+1\}.\end{array}\label{eq:lim-sum-norm-z}
\end{equation}
 Moreover, 
\begin{eqnarray}
|\langle v^{n_{k},\bar{w}}-v^{n_{k},p(n_{k},i)},z_{i}^{n_{k},\bar{w}}\rangle| & \leq & \begin{array}{c}
\|v^{n_{k},\bar{w}}-v^{n_{k},p(n_{k},i)}\|\|z_{i}^{n_{k},\bar{w}}\|\end{array}\label{eq:inn-pdt-sum-norm}\\
 & \leq & \begin{array}{c}
\left(\underset{w=1}{\overset{\bar{w}}{\sum}}\|v^{n_{k},w}-v^{n_{k},w-1}\|\right)\|z_{i}^{n_{k},\bar{w}}\|.\end{array}\nonumber 
\end{eqnarray}
By (ii), there exists a further subsequence of $\{v^{n_{k},\bar{w}}\}_{k=1}^{\infty}$
which converges to some $v^{*}\in X$. Combining \eqref{eq:lim-sum-norm-z}
and \eqref{eq:inn-pdt-sum-norm} gives (iii).

\textbf{Proof of (iv):}  From earlier results, we obtain 
\begin{eqnarray}
 &  & \begin{array}{c}
-\underset{i=1}{\overset{r}{\sum}}h_{i}(x_{0}-v^{*})-\delta_{H^{n_{k},\bar{w}}}(x_{0}-v^{*})\end{array}\label{eq:biggest-formula}\\
 & \overset{\eqref{eq:From-8}}{\leq} & \begin{array}{c}
\frac{1}{2}\|x_{0}-(x_{0}-v^{*})\|^{2}-F_{H^{n_{k},\bar{w}}}(z^{n_{k},\bar{w}})\end{array}\nonumber \\
 & \overset{\scriptsize\mbox{Alg \ref{alg:Ext-Dyk} line 15}}{=} & \begin{array}{c}
\frac{1}{2}\|x_{0}-(x_{0}-v^{*})\|^{2}-F_{H^{n_{k},p(n_{k},r+1)}}(z^{n_{k},\bar{w}})\end{array}\nonumber \\
 & \overset{\eqref{eq:SHQP-dual},\eqref{eq:stagnant-indices}}{=} & \begin{array}{c}
\frac{1}{2}\|v^{*}\|^{2}+\underset{i=1}{\overset{r}{\sum}}h_{i}^{*}(z_{i}^{n_{k},p(n_{k},i)})+\delta_{H^{n_{k},p(n_{k},r+1)}}^{*}(z_{r+1}^{n_{k},\bar{w}})\end{array}\nonumber \\
 &  & \begin{array}{c}
-\langle x_{0},v^{n_{k},\bar{w}}\rangle+\frac{1}{2}\|v^{n_{k},\bar{w}}\|^{2}\end{array}\nonumber \\
 & \overset{\scriptsize\mbox{Claim \ref{claim:Fenchel-duality}(c)},i\in S_{n,p(n,i)}}{=} & \begin{array}{c}
\frac{1}{2}\|v^{*}\|^{2}+\underset{i=1}{\overset{r+1}{\sum}}\langle x_{0}-v^{n_{k},p(n_{k},i)},z_{i}^{n_{k},p(n_{k},i)}\rangle\end{array}\nonumber \\
 &  & \begin{array}{c}
-\underset{i=1}{\overset{r}{\sum}}h_{i}(x_{0}-v^{n_{k},p(n_{k},i)})-\langle x_{0},v^{n_{k},\bar{w}}\rangle+\frac{1}{2}\|v^{n_{k},\bar{w}}\|^{2}\end{array}\nonumber \\
 & \overset{\eqref{eq:stagnant-indices}}{=} & \begin{array}{c}
\frac{1}{2}\|v^{*}\|^{2}-\underset{i=1}{\overset{r+1}{\sum}}\langle v^{n_{k},p(n_{k},i)}-v^{n_{k},\bar{w}},z_{i}^{n_{k},\bar{w}}\rangle\end{array}\nonumber \\
 &  & \begin{array}{c}
-\underset{i=1}{\overset{r}{\sum}}h_{i}(x_{0}-v^{n_{k},p(n_{k},i)})-\underset{i=1}{\overset{r+1}{\sum}}\langle v^{n_{k},\bar{w}},z_{i}^{n_{k},\bar{w}}\rangle+\frac{1}{2}\|v^{n_{k},\bar{w}}\|^{2}\end{array}\nonumber \\
 & \overset{\eqref{eq:from-10}}{=} & \begin{array}{c}
\frac{1}{2}\|v^{*}\|^{2}-\frac{1}{2}\|v^{n_{k},\bar{w}}\|^{2}-\underset{i=1}{\overset{r+1}{\sum}}\langle v^{n_{k},p(n_{k},i)}-v^{n_{k},\bar{w}},z_{i}^{n_{k},\bar{w}}\rangle\end{array}\nonumber \\
 &  & \begin{array}{c}
-\underset{i=1}{\overset{r}{\sum}}h_{i}(x_{0}-v^{n_{k},p(n_{k},i)}).\end{array}\nonumber 
\end{eqnarray}
Since $\lim_{k\to\infty}v^{n_{k},\bar{w}}=v^{*}$, we have $\lim_{k\to\infty}\frac{1}{2}\|v^{*}\|^{2}-\frac{1}{2}\|v^{n_{k},\bar{w}}\|^{2}=0$.
The term $\sum_{i=1}^{r+1}\langle v^{n_{k},p(n_{k},i)}-v^{n_{k},\bar{w}},z_{i}^{n_{k},\bar{w}}\rangle$
converges to 0 by (iii). Next, recall from \eqref{eq:primal-subpblm}
that $x_{0}-v^{n_{k},p(n_{k},i)}\in C_{i}$. Recall from the end of
the proof of (iii) that $x_{0}-v^{*}=\lim_{k\to\infty}x_{0}-v^{n_{k},p(n_{k},i)}$,
so $x_{0}-v^{*}\in C_{i}$. Hence $x_{0}-v^{*}\in\cap_{i=r_{2}+1}^{r}C_{i}$.
Since $H^{n_{k},\bar{w}}$ was designed so that $\cap_{i=r_{2}+1}^{r}C_{i}\subset H^{n_{k},\bar{w}}$,
we have $x_{0}-v^{*}\in H^{n_{k},\bar{w}}$, so $\delta_{H^{n_{k},\bar{w}}}(x_{0}-v^{*})=0$.
Lastly, by the lower semicontinuity of $h_{i}(\cdot)$, we have 
\[
-\lim_{k\to\infty}\sum_{i=1}^{r}h_{i}(x_{0}-v^{n_{k},p(n_{k},i)})\leq-\sum_{i=1}^{r}h_{i}(x_{0}-v^{*}).
\]
Therefore \eqref{eq:biggest-formula} becomes an equation in the limit,
which leads to $\lim_{k\to\infty}F_{H^{n_{k},\bar{w}}}(z^{n_{k},\bar{w}})=\frac{1}{2}\|v^{*}\|^{2}+\sum_{i=1}^{r}h_{i}(x_{0}-v^{*})$.

\end{proof}
We now show some reasonable conditions that guarantee Assumption \ref{assu:Assumptions}(c).
\begin{prop}
\label{prop:satisfy-sqrt-growth}(Satisfying Assumption \ref{assu:Assumptions}(c))
Assumption \ref{assu:Assumptions}(c) is satisfied when all of the
following conditions on $S_{n,w}$ hold:
\begin{enumerate}
\item There are only finitely many $S_{n,w}$ for which $S_{n,w}\cap\{r_{1}+1,\dots,r+1\}$
contains more than one element.
\item There are constants $M_{1}>0$ and $M_{2}>0$ such that the size of
the set 
\[
\big\{(m,w):m\leq n,\,w\in\{1,\dots,\bar{w}\},\,|S_{m,w}|>1\big\}
\]
is bounded by $M_{1}\sqrt{n}+M_{2}$ for all $n$. 
\end{enumerate}
\end{prop}
\begin{proof}
We only need to prove this result for when only condition (2) holds
and $S_{n}\cap\{r_{1}+1,\dots,r+1\}$ always contains at most one
element. We have 
\begin{eqnarray}
\sum_{i=1}^{r+1}\|z_{i}^{n,\bar{w}}\| & \leq & \sum_{i=1}^{r+1}\|z_{i}^{n,0}\|+\sum_{i=1}^{r+1}\sum_{w=1}^{\bar{w}}\|z_{i}^{n,w}-z_{i}^{n,w-1}\|\label{eq:parts}\\
 & \overset{\eqref{eq:aggregate-5}}{\leq} & \sum_{i=1}^{r+1}\|z_{i}^{n-1,\bar{w}}\|+\sum_{i=1}^{r+1}\sum_{w=1}^{\bar{w}}\|z_{i}^{n,w}-z_{i}^{n,w-1}\|.\nonumber 
\end{eqnarray}
Hence 
\[
\sum_{i=1}^{r+1}\|z_{i}^{n,\bar{w}}\|\overset{\eqref{eq:parts}}{\leq}\sum_{i=1}^{r+1}\|z_{i}^{1,0}\|+\sum_{m=1}^{n}\sum_{i=1}^{r+1}\sum_{w=1}^{\bar{w}}\|z_{i}^{n,w}-z_{i}^{n,w-1}\|.
\]
So it suffices to show that there are numbers $A'$ and $B'$ such
that 
\begin{equation}
\sum_{m=1}^{n}\sum_{i=1}^{r+1}\sum_{w=1}^{\bar{w}}\|z_{i}^{m,w}-z_{i}^{m,w-1}\|\leq A'\sqrt{n}+B'.\label{eq:growth-target}
\end{equation}
The sum of the left hand side of \eqref{eq:growth-target} can be
written as 
\begin{equation}
\sum_{(m,w)\in\bar{S}_{n,1}}\sum_{i=1}^{r+1}\|z_{i}^{m,w}-z_{i}^{m,w-1}\|+\sum_{(m,w)\in\bar{S}_{n,2}}\sum_{i=1}^{r+1}\|z_{i}^{m,w}-z_{i}^{m,w-1}\|,\label{eq:break-z-i}
\end{equation}
where\begin{subequations} 
\begin{eqnarray}
\bar{S}_{n,1} & = & \big\{(m,w):|S_{m,w}|=1,\,m\leq n,\,w\in\{1,\dots,\bar{w}\}\big\},\label{eq:S-bar-1}\\
\mbox{ and }\bar{S}_{n,2} & = & \big\{(m,w):|S_{m,w}|>1,\,m\leq n,\,w\in\{1,\dots,\bar{w}\}\big\}.\label{eq:S-bar-2}
\end{eqnarray}
\end{subequations}First, there is a constant $M_{3}$ such that 
\begin{eqnarray}
 &  & \sum_{(m,w)\in\bar{S}_{n,1}}\sum_{i=1}^{r+1}\|z_{i}^{m,w}-z_{i}^{m,w-1}\|\label{eq:2nd-sum}\\
 & \overset{\scriptsize|S_{n,w}|=1\mbox{ in }\eqref{eq:S-bar-1},\eqref{eq:from-10}}{=} & \sum_{(m,w)\in\bar{S}_{n,1}}\|v^{m,w}-v^{m,w-1}\|\nonumber \\
 & \overset{\eqref{eq:S-bar-1}}{\leq} & \sum_{w=1}^{\bar{w}}\sum_{m=1}^{n}\|v^{m,w}-v^{m,w-1}\|\nonumber \\
 & \leq & \sqrt{\bar{w}n}\sqrt{\sum_{w=1}^{\bar{w}}\sum_{m=1}^{n}\|v^{m,w}-v^{m,w-1}\|^{2}}\nonumber \\
 & \overset{\scriptsize\mbox{Thm \ref{thm:convergence}(i)}}{\leq} & \sqrt{n}M_{3}.\nonumber 
\end{eqnarray}
Next, we estimate the second sum in \eqref{eq:break-z-i}. For each
$(m,w)\in\bar{S}_{n,2}$, by condition (1), there is a unique $i_{m,w}\in S_{m,w}\cap\{r_{1}+1,\dots,r+1\}$.
We have 
\begin{equation}
z_{i_{m,w}}^{m,w}-z_{i_{m,w}}^{m,w-1}\overset{\eqref{eq:dual-obj-fn},\eqref{eq:from-10}}{=}v_{i_{m,w}}^{m,w}-v_{i_{m,w}}^{m,w-1}-\sum_{j\in S_{m,w}\backslash\{i_{m,w}\}}(z_{j}^{m,w}-z_{j}^{m,w-1}).\label{eq:bound-z-i}
\end{equation}
For each $j\in S_{m,w}\backslash\{i_{m,w}\}$, we have 
\[
z_{j}^{m,w}\overset{\scriptsize\mbox{Claim \ref{claim:Fenchel-duality}(b)}}{\in}\partial f_{j}(x^{m,w})\overset{\eqref{eq:From-13}}{=}\partial f_{j}(x_{0}-v^{m,w}).
\]
Together with the fact that $v^{m,w}$ is bounded from Theorem \ref{thm:convergence}(ii)
and the fact that $f_{j}(\cdot)$ are Lipschitz on bounded domains,
we deduce that $z_{j}^{m,w}$ and $z_{j}^{m,w-1}$ are bounded for
all $j\in\{1,\dots,r_{1}\}$ by standard convex analysis. Since $S_{m,w}\backslash\{i_{m,w}\}\subset\{1,\dots,r_{1}\}$,
every term on the right hand side of \eqref{eq:bound-z-i} is bounded,
so there is a constant $M_{4}>0$ such that $\|z_{i}^{m}-z_{i}^{m-1}\|\leq M_{4}$.
Therefore condition (2) implies 
\begin{equation}
\sum_{(m,w)\in\bar{S}_{n,2}}\sum_{i=1}^{r+1}\|z_{i}^{m,w}-z_{i}^{m,w-1}\|\leq M_{4}(M_{1}\sqrt{n}+M_{2}).\label{eq:1st-sum}
\end{equation}
Combining \eqref{eq:2nd-sum} and \eqref{eq:1st-sum} into \eqref{eq:break-z-i}
gives the conclusion we need.
\end{proof}

\section{\label{sec:O-1-n-convergence}$O(1/n)$ convergence when a dual minimizer
exists}

In this section, we show that for the problem \eqref{eq:primal},
if Algorithm \ref{alg:Ext-Dyk} is applied with some finite $M$ and
a minimizer for the dual problem exists, then the rate of convergence
of the dual objective function is $O(1/n)$, which leads to the $O(1/\sqrt{n})$
rate of convergence to the primal minimizer.

We recall a lemma on the convergence rates of sequences.
\begin{lem}
\label{lem:seq-conv-rate}(Sequence convergence rate) Let $\alpha>0$.
Suppose the sequence of nonnegative numbers $\{a_{k}\}_{k=0}^{\infty}$
is such that 
\[
a_{k}\geq a_{k+1}+\alpha a_{k+1}^{2}\mbox{ for all }k\in\{1,2,\dots\}.
\]
\end{lem}
\begin{enumerate}
\item \cite[Lemma 6.2]{Beck_Tetruashvili_2013} If furthermore, $\begin{array}{c}
a_{1}\leq\frac{1.5}{\alpha}\mbox{ and }a_{2}\leq\frac{1.5}{2\alpha}\end{array}$, then 
\[
\begin{array}{c}
a_{k}\leq\frac{1.5}{\alpha k}\mbox{ for all }k\in\{1,2,\dots\}.\end{array}
\]

\item \cite[Lemma 3.8]{Beck_alt_min_SIOPT_2015} For any $k\geq2$, 
\[
\begin{array}{c}
a_{k}\leq\max\left\{ \left(\frac{1}{2}\right)^{(k-1)/2}a_{0},\frac{4}{\alpha(k-1)}\right\} .\end{array}
\]
In addition, for any $\epsilon>0$, if 
\[
\begin{array}{c}
\begin{array}{c}
k\geq\max\left\{ \frac{2}{\ln(2)}[\ln(a_{0})+\ln(1/\epsilon)],\frac{4}{\alpha\epsilon}\right\} +1,\end{array}\end{array}
\]
then $a_{n}\leq\epsilon$. 
\end{enumerate}
Instead of condition (A2) after \eqref{eq:primal}, we assume a stronger
condition on $g(\cdot)$:
\begin{itemize}
\item [(A2$'$)]$g_{i}:X\to\mathbb{R}$ are convex functions such that $\mbox{dom}g_{i}(\cdot)$
are open sets for all $i\in\{r_{1}+1,\dots,r_{2}\}$.
\end{itemize}
In other words, the functions $g_{i}(\cdot)$ are such that if $\lim_{j\to\infty}x_{j}$
lies in $\partial\mbox{dom}g_{i}(\cdot)$, then $\lim_{j\to\infty}g_{i}(x_{j})=\infty$.

We have the following theorem.
\begin{thm}
($O(1/n)$ convergence of dual function) Suppose conditions (1) and
(2) in Proposition \ref{prop:satisfy-sqrt-growth} and Assumption
\ref{assu:Assumptions} are satisfied and Algorithm \ref{alg:Ext-Dyk}
is run with finite $M$. If a dual minimizer to \eqref{eq:dual} exists,
then the convergence rate of the dual objective value is $O(1/n)$.
This in turn implies that the convergence rate of $\{\|x^{n,\bar{w}}-x^{*}\|\}_{n}$
is $O(1/\sqrt{n})$. \end{thm}
\begin{proof}
Let $V_{n}=-F_{H^{n,\bar{w}}}(z^{n,\bar{w}})$. Recall that $\{V_{n}\}$
is nonincreasing by Theorem \ref{thm:convergence}(i). We want to
show that $V_{n}-(-\beta)\leq O(1/n)$. 

First, from line 8 of Algorithm \ref{alg:Ext-Dyk}, we have $H^{n,w}\supset\tilde{C}^{n,w+1}$,
so 
\begin{eqnarray*}
 &  & \begin{array}{c}
\frac{1}{2}\|v^{n,w}-x_{0}\|^{2}+\underset{i=1}{\overset{r}{\sum}}h_{i}^{*}(z_{i}^{n,w})+\delta_{H^{n,w}}^{*}(z_{r+1}^{n,w})-\frac{1}{2}\|x_{0}\|^{2}\end{array}\\
 & \overset{H^{n,w}\supset\tilde{C}^{n,w+1}}{\geq} & \begin{array}{c}
\frac{1}{2}\|v^{n,w}-x_{0}\|^{2}+\underset{i=1}{\overset{r}{\sum}}h_{i}^{*}(z_{i}^{n,w})+\delta_{\tilde{C}^{n,w+1}}^{*}(z_{r+1}^{n,w})-\frac{1}{2}\|x_{0}\|^{2}\end{array}\\
 & \overset{\eqref{eq:dual-obj-fn},\eqref{eq:from-10}}{\geq} & \begin{array}{c}
\frac{1}{2}\|v^{n,w+1}-x_{0}\|^{2}+\underset{i=1}{\overset{r}{\sum}}h_{i}^{*}(z_{i}^{n,w+1})+\delta_{\tilde{C}^{n,w+1}}^{*}(z_{r+1}^{n,w})-\frac{1}{2}\|x_{0}\|^{2}\end{array}\\
 &  & \begin{array}{c}
+\frac{1}{2}\|v^{n,w+1}-v^{n,w}\|^{2}\end{array}\\
 & \overset{\scriptsize\mbox{Alg \ref{alg:Ext-Dyk}, line 11}}{=} & \begin{array}{c}
\frac{1}{2}\|v^{n,w+1}-x_{0}\|^{2}+\underset{i=1}{\overset{r}{\sum}}h_{i}^{*}(z_{i}^{n,w+1})+\delta_{H^{n,w+1}}^{*}(z_{r+1}^{n,w})-\frac{1}{2}\|x_{0}\|^{2}\end{array}\\
 &  & \begin{array}{c}
+\frac{1}{2}\|v^{n,w+1}-v^{n,w}\|^{2}.\end{array}
\end{eqnarray*}
In view of the above and the definitions of $F_{H}(\cdot)$ in \eqref{eq:SHQP-dual}
and $V_{n}$, we have 
\begin{equation}
\begin{array}{c}
V_{n}\geq V_{n+1}+\frac{1}{2}\underset{w=1}{\overset{\bar{w}}{\sum}}\|v^{n,w}-v^{n,w-1}\|^{2}.\end{array}\label{eq:dual-val-dec}
\end{equation}
We then look at the subgradients generated in each iteration. Recall
how $z^{n,w}$ were defined in \eqref{eq:dual-obj-fn}. We have, for
each $i\in\{1,\dots r+1\}$, 
\begin{eqnarray}
\overbrace{v^{n,\bar{w}}-v^{n,p(n,i)}}^{s_{i}:=} & \overset{\eqref{eq:From-13}}{=} & v^{n,\bar{w}}-x_{0}+x^{n,p(n,i)}\label{eq:translated-BCD-optim}\\
 & \overset{\scriptsize\mbox{ Claim \ref{claim:Fenchel-duality}(a)}}{\in} & v^{n,\bar{w}}-x_{0}+\partial h_{i}^{*}(z_{i}^{n,p(n,i)})\nonumber \\
 & \overset{\eqref{eq:stagnant-indices}}{=} & v^{n,\bar{w}}-x_{0}+\partial h_{i}^{*}(z_{i}^{n,\bar{w}}).\nonumber 
\end{eqnarray}
Let the vector $s\in X^{r+1}$ be defined so that the $i$th component
$s_{i}\in X$ is as in \eqref{eq:translated-BCD-optim}. Then 
\begin{equation}
\begin{array}{c}
\|s_{i}\|\leq\underset{w=p(n,i)+1}{\overset{\bar{w}}{\sum}}\|v^{n,w}-v^{n,w-1}\|\leq\underset{w=1}{\overset{\bar{w}}{\sum}}\|v^{n,w}-v^{n,w-1}\|.\end{array}\label{eq:subgrad-bdd}
\end{equation}
Let $z^{*}\in X^{r+1}$ be a minimizer of $-F_{H^{n,w}}(\cdot)$ with
$z_{r+1}^{*}=0$. (Such a minimizer can be constructed by appending
$z_{r+1}^{*}=0$ to a minimizer of \eqref{eq:dual}, which was assumed
to exist.) Making use of the elementary fact that $s\in\partial(-F_{H^{n,\bar{w}}})(z^{n,\bar{w}})$,
we have 
\begin{eqnarray}
V_{n}-V^{*} & = & \begin{array}{c}
-F_{H^{n,\bar{w}}}(z^{n,w})-(-F_{H^{n,\bar{w}}}(z^{*}))\end{array}\label{eq:values-for-1-n-bdd}\\
 & \leq & \begin{array}{c}
-\langle s,z^{n,\bar{w}}-z^{*}\rangle\end{array}\nonumber \\
 & \leq & \begin{array}{c}
\underset{i=1}{\overset{r+1}{\sum}}\|s_{i}\|\|z_{i}^{n,\bar{w}}-z_{i}^{*}\|\end{array}\nonumber \\
 & \overset{\eqref{eq:subgrad-bdd}}{\leq} & \begin{array}{c}
\underset{w=1}{\overset{\bar{w}}{\sum}}\|v^{n,w}-v^{n,w-1}\|\underset{i=1}{\overset{r+1}{\sum}}\|z_{i}^{n,\bar{w}}-z_{i}^{*}\|.\end{array}\nonumber 
\end{eqnarray}
\textbf{Claim: There is a constant $M_{4}$ such that $\|z_{i}^{n,w}\|\leq M_{4}$
for all $n\geq0$, $w\in\{1,\dots,\bar{w}\}$ and $i\in\{1,\dots,r+1\}$. }

\textbf{Step 1: The claim is true for all $n\geq0$, $i\in\{1,\dots,r_{2}\}$
and $w\in\{0,\dots,\bar{w}\}$.}

The limit $\lim_{n\to\infty}x^{n,\bar{w}}$ must lie in the interior
of the domains of $f_{i}(\cdot)$ and $g_{i}(\cdot)$ for all $i\in\{1,\dots,r_{2}\}$
(by Assumption (A2$'$)). It is well known that the subgradients of
a convex function is bounded in the interior of its domain, so there
is a constant $M_{1}$ such that $\|z_{i}^{n,w}\|\leq M_{1}$ for
all $i\in\{1,\dots,r_{2}\}$ and $w\in\{0,\dots,\bar{w}\}$. 

\textbf{Step 2: The claim is true for all $n\geq0$, $i\in\{1,\dots,r+1\}$
and $w=0$. }

Since we assumed that Algorithm \ref{alg:Ext-Dyk} was run with a
finite $M$, by \eqref{eq:aggregate-3}, $\|z_{i}^{n,0}\|\leq M$
for all $i\in\{r_{2}+1,\dots,r\}$ and $n\geq0$. Next, we show that
$M_{1}$ can be made larger if necessary so that $\|z_{r+1}^{n,0}\|\leq M_{1}$
for all $n\geq0$. Seeking a contradiction, suppose that there is
a subsequence $\{n_{k}\}$ such that $\lim_{k\to\infty}\|z_{r+1}^{n_{k},0}\|=\infty$.
Then this would mean that $\lim_{k\to\infty}\|\sum_{i=1}^{r+1}z_{i}^{n_{k},0}\|=\infty$.
Recalling \eqref{eq:From-8} for the special case where $x=x^{*}$(the
primal minimizer), we have 
\begin{eqnarray*}
\begin{array}{c}
\frac{1}{2}\Big\| x_{0}-x^{*}-\underset{i=1}{\overset{r+1}{\sum}}z_{i}^{n,0}\Big\|^{2}\end{array} & \leq & \begin{array}{c}
\frac{1}{2}\|x_{0}-x^{*}\|^{2}+\underset{i=1}{\overset{r}{\sum}}h_{i}(x^{*})-F_{H^{n,0}}(z^{n,0})\end{array}\\
 & \overset{\scriptsize\mbox{Thm \ref{thm:convergence}(i)}}{\leq} & \begin{array}{c}
\frac{1}{2}\|x_{0}-x^{*}\|^{2}+\underset{i=1}{\overset{r}{\sum}}h_{i}(x^{*})-F_{H^{1,0}}(z^{1,0}),\end{array}
\end{eqnarray*}
which is a contradiction.

\textbf{Step 3: The claim is true for all $n\geq0$, $i\in\{r_{2}+1,\dots,r+1\}$
and $w\in\{1,\dots,\bar{w}\}$.}

Next, we recall from Theorem \ref{thm:convergence}(i) that $\sum_{n=1}^{\infty}\sum_{w=1}^{\bar{w}}\|v^{n,w}-v^{n,w-1}\|^{2}$
is finite. This implies that there is a $M_{2}>0$ such that $\|v^{n,w}-v^{n,w-1}\|\leq M_{2}$
for all $n\geq0$ and $w\in\{1,\dots,\bar{w}\}$. 

Next, for each $n\geq0$ and $i\in\{r_{2}+1,\dots,r\}$, we want to
show that there is a constant $M_{3}$ such that $\|z_{i}^{n,w}\|\leq M_{3}$
for all $n\geq0$ and $w\in\{1,\dots,\bar{w}\}$. Since the $z_{i}^{n,w}$
were chosen by condition (1) in Proposition \ref{prop:satisfy-sqrt-growth},
then if $n$ is large enough, if $S_{n,w}\cap\{r_{2}+1,\dots,r+1\}\neq\emptyset$,
then there is a $i_{n,w}\in S_{n,w}$ such that $S_{n,w}\backslash\{i_{n,w}\}\subset\{1,\dots,r_{2}\}$.
We have 
\begin{equation}
z_{i_{n,w}}^{n,w}\overset{\eqref{eq:from-10}}{=}z_{i_{n,w}}^{n,w-1}+v^{n,w}-v^{n,w-1}-\sum_{j\in S_{n,w}\backslash\{i_{n,w}\}}[z_{j}^{n,w}-z_{j}^{n,w-1}].\label{eq:rewrite-z}
\end{equation}
Then we have 
\begin{eqnarray*}
\|z_{i_{n,w}}^{n,w}\| & \overset{\eqref{eq:rewrite-z}}{\leq} & \|z_{i_{n,w}}^{n,w-1}\|+\|v^{n,w}-v^{n,w-1}\|\\
 &  & +\sum_{j\in S_{n,w}\backslash\{i_{n,w}\}}[\|z_{j}^{n,w}\|+\|z_{j}^{n,w-1}\|]\\
 & \overset{S_{n,w}\backslash\{i_{n,w}\}\subset\{1,\dots,r_{2}\}\scriptsize\mbox{, Step 1}}{\leq} & \|z_{i_{n,w}}^{n,w-1}\|+M_{2}+2r_{2}M_{1}.
\end{eqnarray*}
This would easily imply that $\|z_{i}^{n,w}\|\leq M_{4}$ for some
$M_{4}>0$ for all $n\geq0$, $i\in\{1,\dots,r+1\}$, and $w\in\{1,\dots,\bar{w}\}$
as needed, ending the proof of the claim.

Now, 
\begin{equation}
\begin{array}{c}
\underset{w=1}{\overset{\bar{w}}{\sum}}\|v^{n,w}-v^{n,w-1}\|\leq\sqrt{2\bar{w}}\sqrt{\frac{1}{2}\underset{w=1}{\overset{\bar{w}}{\sum}}\|v^{n,w}-v^{n,w-1}\|^{2}}\overset{\eqref{eq:dual-val-dec}}{\leq}\sqrt{2\bar{w}}\sqrt{V_{n}-V_{n+1}}.\end{array}\label{eq:sum-norms}
\end{equation}
Then combining the above, we have 
\begin{eqnarray}
V_{n}-V^{*} & \overset{\eqref{eq:values-for-1-n-bdd},\eqref{eq:sum-norms}}{\leq} & \left[\sum_{i=1}^{r+1}\|z_{i}^{n,\bar{w}}\|+\sum_{i=1}^{r+1}\|z_{i}^{*}\|\right]\sqrt{2\bar{w}}\sqrt{V_{n}-V_{n+1}}\label{eq:V-parent-form}\\
 & \overset{\scriptsize\mbox{by earlier claim}}{\leq} & \left[(r+1)M_{4}+\sum_{i=1}^{r+1}\|z_{i}^{*}\|\right]\sqrt{2\bar{w}}\sqrt{V_{n}-V_{n+1}}.\nonumber 
\end{eqnarray}
Letting $M_{5}$ be $(r+1)M_{4}+\sum_{i=1}^{r+1}\|z_{i}^{*}\|$ and
rearranging \eqref{eq:V-parent-form}, we have 
\[
\begin{array}{c}
V_{n}-V^{*}\geq V_{n+1}-V^{*}+\frac{1}{2\bar{w}M_{5}^{2}}(V_{n+1}-V^{*})^{2}.\end{array}
\]
Applying Lemma \ref{lem:seq-conv-rate} gives the first statement
of our conclusion. The second statement comes from substituting $x=x^{*}$
in \eqref{eq:From-8} and noticing that $x_{0}-x^{*}-\sum_{i=1}^{r+1}z_{i}\overset{\eqref{eq_m:from-10-13}}{=}x^{n,\bar{w}}-x^{*}.$\end{proof}
\begin{rem}
(Nonexistence of dual minimizers) An example of a problem where dual
minimizers do not exist is in \cite[page 9]{Han88}. Lemma 2 in \cite{Gaffke_Mathar}
shows that a fast convergence rate to the primal minimizer implies
the existence of dual minimizers.
\end{rem}

\section{\label{sec:approx-prox}Approximate proximal point algorithm}

Consider the problem of minimizing 
\begin{equation}
\begin{array}{c}
\underset{i=1}{\overset{r}{\sum}}h_{i}(x).\end{array}\label{eq:basic-convex}
\end{equation}
If one of the functions $h_{i}(\cdot)$ can be split as $h_{i}(\cdot)=\tilde{h}_{i}(\cdot)+\frac{c_{i}}{2}\|\cdot-x_{0}\|^{2}$
for some convex $\tilde{h}_{i}(\cdot)$ and $c_{i}>0$, then \eqref{eq:basic-convex}
can be minimized using Dykstra's splitting algorithm of Section \ref{sec:Dykstra-splitting}.
In this section, we propose an approximate proximal point method for
minimizing \eqref{eq:basic-convex} without splitting $h_{i}(\cdot)$.
We first present Algorithm \ref{alg:approx-prox-point} and prove
that all its cluster points are minimizers of the parent problem.
Then, in Subsection \ref{sub:Satisfy-approx-min}, we show that the
Dykstra splitting investigated in Section \ref{sec:Dykstra-splitting}
can find an approximate primal minimizer required in Algorithm \ref{alg:approx-prox-point}.

\subsection{\label{sub:approx-prox-point}An approximate proximal point algorithm}

Consider the problem of minimizing $h:X\to\mathbb{R}$, where 
\begin{equation}
\begin{array}{c}
h(\cdot)=\delta_{D}(\cdot)+\underset{i=1}{\overset{r_{2}}{\sum}}h_{i}(\cdot),\end{array}\label{eq:approx-point-form}
\end{equation}
and each $h_{i}:X\to\mathbb{R}$ is a closed convex function whose
domain is an open set, and $D$ is a compact convex set in $X$. This
setting is less general than that of \eqref{eq:primal}, since it
does not allow for all lower semicontinuous convex functions, and
we only allow for one compact set $D$ instead of $r-r_{2}$ sets. 

Algorithm \ref{alg:approx-prox-point} shows an approximate proximal
point algorithm, where one solves a regularized version of \eqref{eq:approx-point-form}
and shifts the proximal center $x_{k}$ when an approximate KKT condition
is satisfied.

\begin{algorithm}[h]
\begin{lyxalgorithm}
\label{alg:approx-prox-point}(Approximate proximal point algorithm)
Consider the problem of minimizing $h(\cdot)$ of the form \eqref{eq:approx-point-form}.
Let $\{\delta_{j}\}_{j=1}^{\infty}\subset\mathbb{R}$ be a sequence
such that $\lim_{j\to\infty}\delta_{j}=0$. Let $x_{0}\in X$. Our
algorithm is as follows:

For $j=1,\dots$

$\qquad$(Find an approximate minimizer $x_{j}$ of $\min\delta_{D}(\cdot)\!+\!\frac{1}{2}\|\cdot-x_{j-1}\|^{2}\!+\!\sum_{i=1}^{r_{2}}h_{i}(\cdot)$.)

$\qquad$Specifically, find $x_{j}\in X$, $z^{(j)},e^{(j)}\in X^{r_{2}+1}$
and \\
$\qquad\qquad$a closed convex set $D^{j}\supset D$ such that\begin{subequations}\label{eq_m:approx-min}
\begin{eqnarray}
\begin{array}{c}
z_{i}^{(j)}\end{array} & \in & \partial h_{i}(x_{j}+e_{i}^{(j)})\mbox{ for all }i\in\{1,\dots,r_{2}\},\label{eq:approx-min-1}\\
\begin{array}{c}
z_{0}^{(j)}\end{array} & \in & N_{D^{j}}(x_{j}+e_{0}^{(i)}),\label{eq:approx-min-2}\\
\begin{array}{c}
\left\Vert (x_{j}-x_{j-1})+\underset{i=0}{\overset{r_{2}}{\sum}}z_{i}^{(j)}\right\Vert \end{array} & \leq & \delta_{j},\label{eq:approx-min-3}\\
\begin{array}{c}
\|e_{i}^{(j)}\|\end{array} & \leq & \delta_{j}\mbox{ for all }i\in\{0,\dots,r_{2}\}.\label{eq:approx-min-4}
\end{eqnarray}

\end{subequations}end For.
\end{lyxalgorithm}
\end{algorithm}

If $r_{2}=1$, $D=X$ and $h_{1}(\cdot)$ were allowed to be any lower
semicontinuous convex function, then Algorithm \ref{alg:approx-prox-point}
would resemble the classical proximal point algorithm. 

Define the operator $T:X\to X$ by 
\begin{equation}
\begin{array}{c}
T(x):=\mbox{prox}_{h}(x):=\underset{x'}{\arg\min}\,\,h(x')+\frac{1}{2}\|x'-x\|^{2}.\end{array}\label{eq:def-T}
\end{equation}
 This operator has some favorable properties in monotone operator
theory. 

We prove our first result. 
\begin{lem}
\label{lem:approx-of-T}(Approximate of $T(\cdot)$) Consider the
problem \eqref{eq:approx-point-form}. Let $T(\cdot)$ be as defined
in \eqref{eq:def-T}. Suppose $D\subset X$ is compact and convex.
For all $\epsilon>0$, there is a $\delta>0$ such that for all $x,x^{+}\in X$
and $z,e\in X^{r+1}$ such that $d(x,D)\leq\delta$ and 
\begin{eqnarray*}
\begin{array}{c}
z_{i}\end{array} & \in & \partial h_{i}(x^{+}+e_{i})\mbox{ for all }i\in\{1,\dots,r_{2}\},\\
\begin{array}{c}
z_{0}\end{array} & \in & N_{\tilde{D}}(x^{+}+e_{0}),\\
\begin{array}{c}
\Big\|(x^{+}-x)+\underset{i=0}{\overset{r_{2}}{\sum}}z_{i}\Big\|\end{array} & \leq & \delta,\\
\begin{array}{c}
\|e_{i}\|\end{array} & \leq & \delta\mbox{ for all }i\in\{0,\dots,r_{2}\},
\end{eqnarray*}
where $\tilde{D}\supset D$ is a closed convex set, then $\|x^{+}-T(x)\|\leq\epsilon$.\end{lem}
\begin{proof}
Seeking a contradiction, suppose otherwise. Then there exists a $\epsilon>0$
such that for all positive integers $k$, there are $x_{k},x_{k}^{+}\in X$,
$z^{(k)},e^{(k)}\in X^{r+1}$ and a closed convex set $D^{(k)}\supset D$
such that\begin{subequations}\label{eq_m:approx-T} 
\begin{eqnarray}
d(x_{k},D) & \leq & 1/k,\label{eq:approx-T-1}\\
z_{i}^{(k)} & \in & \partial h_{i}(x_{k}^{+}+e_{i}^{(k)}),\nonumber \\
z_{0}^{(k)} & \in & N_{D^{(k)}}(x_{k}^{+}+e_{0}^{(k)}),\nonumber \\
\Big\|\underbrace{(x_{k}^{+}-x_{k})+\sum_{i=0}^{r_{2}}z_{i}^{(k)}}_{d_{k}}\Big\| & \leq & 1/k,\label{eq:approx-T-4}\\
\|e_{i}^{(k)}\| & \leq & 1/k\mbox{ for all }i\in\{0,\dots,r_{2}\},\nonumber 
\end{eqnarray}
\end{subequations}but 
\begin{equation}
\|x_{k}^{+}-T(x_{k})\|\geq\epsilon.\label{eq:z-prox-z-geq-epsilon}
\end{equation}
Letting $k\nearrow\infty$, we can assume (by taking subsequences
if necessary) that 
\begin{equation}
\lim_{k\to\infty}x_{k}=\bar{x}\mbox{, }\lim_{k\to\infty}x_{k}^{+}=\bar{x}^{+}\mbox{ and }\lim_{k\to\infty}e^{(k)}=0.\label{eq:some-lims}
\end{equation}
There are two cases we need to consider.

\textbf{Case 1: $\bar{x}^{+}$ lies in the interior of $\mbox{dom}h_{i}$
for all $i\in\{1,\dots,r_{2}\}$.}

Making use of the fact that convex functions are locally Lipschitz
in the interior of their domains, we obtain the boundedness of $\{z^{(k)}\}$.
We can assume (by taking subsequences if necessary) that $\lim_{k\to\infty}z^{(k)}=\bar{z}$.
Taking the limits of \eqref{eq_m:approx-T} as $k\to\infty$ would
give us $\bar{z}_{i}\in\partial h_{i}(\bar{x}^{+})$, $\bar{z}_{0}\in N_{D}(\bar{x}^{+})$
and $\bar{x}^{+}-\bar{x}+\sum_{i=0}^{r_{2}}\bar{z}_{i}=0$, which
would in turn imply that $\bar{x}^{+}=T(\bar{x})$. It is well known
that $T(\cdot)$ is nonexpansive and hence continuous, so 
\[
0<\epsilon\overset{\eqref{eq:z-prox-z-geq-epsilon}}{\leq}\lim_{k\to\infty}\|x_{k}^{+}-T(x_{k})\|=\|\bar{x}^{+}-T(\bar{x})\|=0,
\]
a contradiction. 

\textbf{Case 2: $\bar{x}^{+}$ lies on the boundary of $\mbox{dom}h_{i}$
for some $i\in\{1,\dots,r_{2}\}$.}

We cannot use the method in Case 1 as some components of $\{z^{(k)}\}$
might be unbounded. We now consider the perturbed functions $h_{i,k}(\cdot)$
defined by 
\begin{equation}
h_{i,k}(x):=h_{i}(x+e_{i}^{(k)}).\label{eq:perturbed-h}
\end{equation}
Let $\tilde{h}_{k}:X\to\mathbb{R}$ be defined by $\tilde{h}_{k}(\cdot)=\delta_{D^{(k)}}(\cdot)+\sum_{i=1}^{r_{2}}h_{i,k}(\cdot)$.
Then the conditions \eqref{eq_m:approx-min} imply that 
\begin{equation}
\begin{array}{c}
x_{k}^{+}=\mbox{prox}_{\tilde{h}_{k}}(x_{k}+d_{k})=\arg\underset{x}{\min}\,\tilde{h}_{k}(x)+\frac{1}{2}\|x-(x_{k}+d_{k})\|^{2},\end{array}\label{eq:prox-arg-min}
\end{equation}
where $d_{k}$ is marked in \eqref{eq:approx-T-4}. Suppose $\bar{i}$
is such that $\bar{x}^{+}$ lies on the boundary of $\mbox{dom}h_{\bar{i}}$.
Then we have that $\lim_{k\to\infty}h_{\bar{i}}(x_{k}^{+})=\infty$.
Since $D$ is bounded, $\inf_{x\in D}h_{i}(x)$ is a finite number
for all $i$, which implies that 
\begin{equation}
\begin{array}{c}
\underset{k\to\infty}{\lim}\tilde{h}_{k}(x_{k}^{+})+\frac{1}{2}\|x_{k}^{+}-(x_{k}+d_{k})\|^{2}=\infty.\end{array}\label{eq:goes-to-infty}
\end{equation}
Next, let $x'_{k}=\mbox{prox}_{h}(x_{k})$ and $x'=\mbox{prox}_{h}(\bar{x})$.
By the continuity properties of $T(\cdot)=\mbox{prox}_{h}(\cdot)$
and $\lim_{k\to\infty}x_{k}=\bar{x}$ in \eqref{eq:some-lims}, we
must have $\lim_{k\to\infty}x'_{k}=x'$. It is clear that $x'$ lies
in $\mbox{dom}(h_{i})$. Since we assumed that $\mbox{dom}(h_{i})$
is open, $x'\in\mbox{int\,}\mbox{dom}(h_{i})$ for all $i$. We then
have 
\begin{equation}
\begin{array}{c}
\underset{k\to\infty}{\lim}\tilde{h}_{k}(x')+\frac{1}{2}\|x'-x_{k}\|^{2}=h(x')+\frac{1}{2}\|x'-\bar{x}\|^{2}<\infty.\end{array}\label{eq:contrad-1}
\end{equation}
 But on the other hand, since $\lim_{k\to\infty}d_{k}=0$, we have
\begin{equation}
\begin{array}{c}
\underset{k\to\infty}{\lim}\tilde{h}_{k}(x')+\frac{1}{2}\|x'-(x_{k}+d_{k})\|^{2}\overset{\eqref{eq:prox-arg-min}}{\geq}\underset{k\to\infty}{\lim}\tilde{h}_{k}(x_{k}^{+})+\frac{1}{2}\|x_{k}^{+}-(x_{k}+d_{k})\|^{2}\overset{\eqref{eq:goes-to-infty}}{=}\infty.\end{array}\label{eq:contrad-2}
\end{equation}
Formulas \eqref{eq:contrad-1} and \eqref{eq:contrad-2} are contradictory,
so $\bar{x}^{+}$ must lie in the interior of all $\mbox{dom}h_{i}$
for all $i$, which reduces to case 1.

Thus we are done.
\end{proof}
To simplify notation in the next two results, we define the set $A$
to be 
\[
A:=\arg\min_{x}h(x).
\]
 We have another lemma.
\begin{lem}
\label{lem:T-diff-implies-dist-small}For all $\epsilon>0$, there
exists $\delta>0$ such that for all $w$ such that $d(w,D)\leq\delta$,
we have 
\[
\|T(w)-w\|\leq\delta\mbox{ implies }d(w,A)\leq\epsilon.
\]
\end{lem}
\begin{proof}
Seeking a contradiction, suppose otherwise. In other words, there
is a $\bar{\epsilon}>0$ such that for all $k>0$, there is a $w_{k}$
such that $d(w_{k},A)>\bar{\epsilon}$ but $\|T(w_{k})-w_{k}\|\leq\frac{1}{k}$.
By taking subsequences if necessary, let $\bar{w}=\lim_{k\to\infty}w_{k}$,
which exists by the compactness of $D$. Taking limits as $k\nearrow\infty$
gives us $\bar{w}=T(\bar{w})$, which will in turn imply that $d(\bar{w},A)=0$,
a contradiction.\end{proof}
\begin{thm}
\label{thm:cluster-pts-good}(Cluster points of Algorithm \ref{alg:approx-prox-point})
All cluster points of $\{x_{j}\}$ in Algorithm \ref{alg:approx-prox-point}
are minimizers of $h(\cdot)$. \end{thm}
\begin{proof}
For any $\epsilon_{1}>0$, we make use of Lemma \ref{lem:T-diff-implies-dist-small}
and obtain $\delta_{1}>0$ such that 
\begin{equation}
\|T(w)-w\|\leq\delta_{1}\mbox{ implies }d(w,A)\leq\epsilon_{1}.\label{eq:T-condn-implies-d}
\end{equation}
By Lemma \ref{lem:approx-of-T}, there exists some $K$ large enough
so that 
\begin{equation}
\|x_{k+1}-T(x_{k})\|\leq\epsilon\mbox{ for all }k\geq K.\label{eq:estimates-good}
\end{equation}
Let $\epsilon>0$ be small enough so that $\frac{4\epsilon^{2}+\delta_{1}^{2}}{4\epsilon}>\mbox{diam}(D)$.
Then since $\lim_{k\to\infty}d(x_{k},D)=0$ and $A\subset D$, we
can increase $K$ if necessary so that 
\begin{equation}
\begin{array}{c}
d(x_{k},A)\leq\mbox{diam}(D)<\frac{4\epsilon^{2}+\delta_{1}^{2}}{4\epsilon}\mbox{ for all }k\geq K.\end{array}\label{eq:jumble}
\end{equation}
Let $\bar{x}_{k}=P_{A}(x_{k})$ so that $d(x_{k},A)=\|x_{k}-\bar{x}_{k}\|$.
It is well known from the theory of monotone operators that $T(\cdot)$
is firmly nonexpansive (see for example \cite[Definition 4.1(i), Proposition 12.27]{BauschkeCombettes11}),
so we have 
\begin{equation}
\|T(x_{k})-\bar{x}_{k}\|^{2}+\|x_{k}-T(x_{k})\|^{2}\leq\|x_{k}-\bar{x}_{k}\|^{2}.\label{eq:firmly-nonexp}
\end{equation}
Suppose $k\geq K$. We split our analysis into two cases.

\textbf{Case 1: $d(x_{k},A)>\epsilon_{1}$. }

Then \eqref{eq:T-condn-implies-d} implies $\|T(x_{k})-x_{k}\|>\delta_{1}$.
We have 
\begin{eqnarray}
d(x_{k+1},A) & \overset{\bar{x}_{k}\in A}{\leq} & \|x_{k+1}-\bar{x}_{k}\|\label{eq:dist-go-down}\\
 & \leq & \|x_{k+1}-T(x_{k})\|+\|T(x_{k})-\bar{x}_{k}\|\nonumber \\
 & \overset{\eqref{eq:estimates-good},\eqref{eq:firmly-nonexp}}{\leq} & \epsilon+\sqrt{\|x_{k}-\bar{x}_{k}\|^{2}-\|x_{k}-T(x_{k})\|^{2}}\nonumber \\
 & < & \epsilon+\sqrt{d(x_{k},A)^{2}-\delta_{1}^{2}}\nonumber \\
 & \overset{\scriptsize\mbox{rearrange }\eqref{eq:jumble}}{\leq} & d(x_{k},A)-\epsilon.\nonumber 
\end{eqnarray}

\textbf{Case 2: $d(x_{k},A)\leq\epsilon_{1}$. }

We have 
\begin{eqnarray*}
d(x_{k+1},A)\overset{\bar{x}_{k}\in A}{\leq}\|x_{k+1}-\bar{x}_{k}\| & \leq & \|x_{k+1}-T(x_{k})\|+\|T(x_{k})-\bar{x}_{k}\|\\
 & \overset{\eqref{eq:estimates-good},\eqref{eq:firmly-nonexp}}{\leq} & \epsilon+\|x_{k}-\bar{x}_{k}\|=\epsilon+d(x_{k},A)\leq\epsilon+\epsilon_{1}.
\end{eqnarray*}
The analysis in these two cases implies $d(x_{k+1},A)\leq\epsilon_{1}+\epsilon$
for all $k$ large enough. Since $\epsilon_{1}$ and $\epsilon$ can
be made arbitrarily small, any cluster point $x'$ of $\{x_{k}\}_{k=1}^{\infty}$
must thus satisfy $d(x',A)=0$, or $x'\in A$. 
\end{proof}

\subsection{\label{sub:Satisfy-approx-min}Satisfying \eqref{eq_m:approx-min}
using Dykstra splitting }

Consider the problem of minimizing $h:X\to\mathbb{R}$, where
\begin{equation}
\begin{array}{c}
h(x)=\underset{i=1}{\overset{r_{2}}{\sum}}h_{i}(x)+\underset{i=r_{2}+1}{\overset{r}{\sum}}\delta_{C_{i}}(x).\end{array}\label{eq:multiple-set-primal}
\end{equation}
and each $h_{i}:X\to\mathbb{R}$ is a closed convex function whose
domain is an open set, and each $C_{i}$ is a closed convex set such
that $\cap_{i=r_{2}+1}^{r}C_{i}$ is compact. This formulation is
slightly more general than that of \eqref{eq:approx-point-form}.
Theorem \ref{thm:shift-prox-center} below shows that Algorithm \ref{alg:Ext-Dyk}
can find approximate minimizers to \eqref{eq:multiple-set-primal}
that will satisfy the conditions for moving to a new proximal center
in Algorithm \ref{alg:approx-prox-point}. 
\begin{thm}
\label{thm:shift-prox-center} Consider the problem \eqref{eq:multiple-set-primal}.
For any $\delta>0$, there is some $n>0$ such that when Algorithm
\ref{alg:Ext-Dyk} is applied to solve 
\begin{equation}
\begin{array}{c}
h(x)=\underset{i=1}{\overset{r_{2}}{\sum}}h_{i}(x)+\underset{i=r_{2}+1}{\overset{r}{\sum}}\delta_{C_{i}}(x)+\frac{1}{2}\|x-x_{0}\|^{2},\end{array}\label{eq:subproblem}
\end{equation}
we have a set $D^{(n)}$, and points $x^{(n)}\in X$ and $\tilde{z}\in X$
such that\begin{subequations}\label{eq_m:Dyk-spl-can} 
\begin{eqnarray}
\begin{array}{c}
\|x^{n,\bar{w}}-x^{n,w}\|\end{array} & \leq & \delta\mbox{ for all }w\in\{1,\dots,\bar{w}\},\label{eq:Dyk-spl-can-1}\\
\begin{array}{c}
z_{i}^{n,\bar{w}}\end{array} & \in & \partial h_{i}(x^{n,p(n,i)})\mbox{ for all }i\in\{1,\dots,r\},\label{eq:Dyk-spl-can-2}\\
\!\!\!\!\!\!\!\!\begin{array}{c}
\|\underset{i=1}{\overset{r_{2}}{\sum}}z_{i}^{n,\bar{w}}+\tilde{z}+x^{n,\bar{w}}-x_{0}\|\end{array} & \leq & \delta,\label{eq:Dyk-spl-can-3}\\
\begin{array}{c}
\cap_{i=r_{2}+1}^{r}C_{i}\end{array} & \subset & D^{(n)},\label{eq:Dyk-spl-can-4}\\
\begin{array}{c}
\tilde{z}\end{array} & \in & N_{D^{(n)}}(x^{(n)}),\label{eq:Dyk-spl-can-5}\\
\begin{array}{c}
\|x^{(n)}-x^{n,\bar{w}}\|\end{array} & \leq & \delta.\label{eq:Dyk-spl-can-6}
\end{eqnarray}
\end{subequations}\end{thm}
\begin{proof}
Define $z_{0}^{n,\bar{w}}$ to be 
\begin{equation}
\begin{array}{c}
z_{0}^{n,\bar{w}}=\underset{i=r_{2}+1}{\overset{r+1}{\sum}}z_{i}^{n,\bar{w}}.\end{array}\label{eq:sum-z}
\end{equation}
Theorem \ref{thm:convergence} says that for any $\delta>0$, we can
find $n>0$ such that the first 3 conditions in \eqref{eq_m:Dyk-spl-can}
hold if $\tilde{z}$ were chosen to be $z_{0}^{n,\bar{w}}$. We separate
into two cases, and discuss how the set $D^{(n)}$ (which will actually
be either the whole space $X$ or a halfspace) and the point $x^{(n)}$
are constructed. 

\textbf{Case 1: $\liminf_{n\to\infty}\|z_{0}^{n,\bar{w}}\|=0$.}

By taking subsequences $\{n_{k}\}_{k}$, we can assume that $\lim_{k\to\infty}\|z_{0}^{n_{k},\bar{w}}\|=0$.
Then the set $D^{(n)}$ can be chosen to be $X$, and $x^{(n)}$ can
be chosen to be $x^{*}$, the minimizer of \eqref{eq:subproblem}.
The vector $\tilde{z}$ can be chosen to be zero, and the inequalities
in \eqref{eq_m:Dyk-spl-can} can be easily seen to be satisfied.

\textbf{Case 2: $\liminf_{n\to\infty}\|z_{0}^{n,\bar{w}}\|>0$.}

Recall that for $i\in\{r_{2}+1,\dots,r\}$, the dual vector $z_{i}^{n,\bar{w}}$
was constructed so that $z_{i}^{n,\bar{w}}\in N_{C_{i}}(x^{n,p(n,i)})$,
and that $z_{r+1}^{n,\bar{w}}\in N_{H^{n,\bar{w}}}(x^{n,p(n,r+1)})$.
For $i\in\{r_{2}+1,\dots,r+1\}$, define the halfspace $\tilde{H}^{n,i}$
to be the halfspace with $x^{n,p(n,i)}$ on its boundary and outward
normal vector $z_{i}^{n,\bar{w}}$ like in \eqref{eq:derived-halfspace}.
Note that $\tilde{H}^{n,i}\supset C_{i}$ for $i\in\{r_{2}+1,\dots,r\}$,
and $\tilde{H}^{n,r+1}\supset H^{n,\bar{w}}$. 

Let $D^{(n)}$ be the halfspace with outward normal $z_{0}^{n,\bar{w}}$
such that $D^{(n)}\supset\bigcap_{i=r_{2}+1}^{r+1}\tilde{H}^{n,i}$,
and $\partial D^{(n)}\cap\bigcap_{i=r_{2}+1}^{r+1}\tilde{H}^{n,i}\neq\emptyset$.
Through Fact \ref{fact:agg-halfspaces}, this choice of $D^{(n)}$
would give us 
\begin{equation}
\begin{array}{c}
\delta_{D^{(n)}}^{*}(z_{0}^{n,\bar{w}})\leq\underset{i=r_{2}+1}{\overset{r}{\sum}}\delta_{C_{i}}^{*}(z_{i}^{n,\bar{w}})+\delta_{H^{n,\bar{w}}}^{*}(z_{r+1}^{n,\bar{w}}).\end{array}\label{eq:aggregate-dual}
\end{equation}
We now show how to satisfy \eqref{eq:Dyk-spl-can-5} and \eqref{eq:Dyk-spl-can-6}
with $\tilde{z}=z_{0}^{n,\bar{w}}$. Let $x^{*}$ be the optimal primal
solution of \eqref{eq:subproblem}. We now want to show that we can
choose a further subsequence if necessary so that $\lim_{k\to\infty}d(\partial D^{(n_{k})},x^{*})=0$.
From the definition of the support function and the fact that $x^{*}\in\bigcap_{i=1}^{r}C_{i}\subset D^{(n)}$,
we have 

\begin{equation}
\begin{array}{c}
\delta_{D^{(n)}}^{*}(z_{0}^{n,\bar{w}})-\langle x^{*},z_{0}^{n,\bar{w}}\rangle=\delta_{D^{(n)}-x^{*}}^{*}(z_{0}^{n,\bar{w}})=\|z_{0}^{n,\bar{w}}\|d(\partial D^{(n)},x^{*}).\end{array}\label{eq:translate}
\end{equation}
We now mimic \eqref{eq:From-8} to obtain 
\begin{eqnarray}
 &  & \begin{array}{c}
\frac{1}{2}\|x_{0}-x^{*}\|^{2}+\underset{i=1}{\overset{r_{2}}{\sum}}h_{i}(x^{*})+\underset{i=r_{2}+1}{\overset{r}{\sum}}\delta_{C_{i}}(x^{*})-F_{H^{n,\bar{w}}}(z^{n,\bar{w},},z_{r+1}^{n,\bar{w}})\end{array}\nonumber \\
 & \overset{\eqref{eq:SHQP-dual}}{=} & \begin{array}{c}
\frac{1}{2}\|x_{0}-x^{*}\|^{2}+\underset{i=1}{\overset{r_{2}}{\sum}}[h_{i}(x^{*})+h_{i}^{*}(z_{i}^{n,\bar{w}})]-\langle x_{0},\underset{i=1}{\overset{r}{\sum}}z_{i}^{n,\bar{w}}\rangle\end{array}\label{eq:long-derv}\\
 &  & \begin{array}{c}
+\frac{1}{2}\|\underset{i=1}{\overset{r}{\sum}}z_{i}^{n,\bar{w}}\|^{2}+\underset{i=r_{2}+1}{\overset{r}{\sum}}\delta_{C_{i}}^{*}(z_{i}^{n,\bar{w}})+\delta_{H^{n,\bar{w}}}^{*}(z_{r+1}^{n,\bar{w}})\end{array}\nonumber \\
 & \overset{\scriptsize\mbox{Fenchel duality, }\eqref{eq:aggregate-dual}}{\geq} & \begin{array}{c}
\frac{1}{2}\|x_{0}-x^{*}\|^{2}+\bigg[\underset{i=1}{\overset{r}{\sum}}\langle x^{*},z_{i}^{n,\bar{w}}\rangle-\underset{i=r_{2}+1}{\overset{r}{\sum}}\langle x^{*},z_{i}^{n,\bar{w}}\rangle\bigg]\end{array}\nonumber \\
 &  & \begin{array}{c}
-\langle x_{0},\underset{i=1}{\overset{r}{\sum}}z_{i}^{n,\bar{w}}\rangle+\frac{1}{2}\|\underset{i=1}{\overset{r}{\sum}}z_{i}^{n,\bar{w}}\|^{2}+\delta_{D^{(n)}}^{*}(z_{0}^{n,\bar{w}})\end{array}\nonumber \\
 & \overset{\eqref{eq:sum-z},\eqref{eq:translate}}{=} & \begin{array}{c}
\frac{1}{2}\|\underset{i=1}{\overset{r}{\sum}}z_{i}^{n,\bar{w}}+x^{*}-x_{0}\|^{2}+\|z_{0}^{n,\bar{w}}\|d(\partial D^{(n)},x^{*})\geq0.\end{array}\nonumber 
\end{eqnarray}
Theorem \ref{thm:convergence}(i)(iv) shows that the formula in the
first line of \eqref{eq:long-derv} has a limit of zero, so by the
squeeze theorem, the last term in \eqref{eq:long-derv} also has limit
zero as $n\to\infty$. This means that $\lim_{k\to\infty}d(\partial D^{(n_{k})},x^{*})=0$.
Thus $x^{(n)}$ can be chosen as the projection of $x^{*}$ onto $\partial D^{(n)}$.
We then have $\lim_{k\to\infty}\|x^{(n)}-x^{*}\|=0$, and so the inequalities
in \eqref{eq_m:Dyk-spl-can} can be easily seen to be satisfied.\end{proof}
\begin{rem}
(Achieving \eqref{eq:approx-min-2}) A final detail is to find an
implementable way for checking that $d(x^{(n)},\cap_{i=r_{2}+1}^{r}C_{i})$
is indeed small. Note that $x^{n,p(n,i)}\in C_{i}$ for all $i\in\{r_{2}+1,\dots,r\}$,
so we have the estimate $d(x^{n,\bar{w}},C_{i})\leq\|x^{n,\bar{w}}-x^{n,p(n,i)}\|$.
One can easily make use of the compactness of $\cap_{i=r_{2}+1}^{r}C_{i}$
to prove that for all $\delta_{1}>0$, there is a $\delta_{2}>0$
such that 
\[
\max_{r_{2}+1\leq i\leq r}d(x,C_{i})\leq\delta_{2}\mbox{ implies }d(x,\cap_{i=r_{2}+1}^{r}C_{i})\leq\delta_{1}.
\]
Putting this fact together with Theorem \ref{thm:shift-prox-center}
allows us to find a point satisfying \eqref{eq_m:approx-min} in Algorithm
\ref{alg:approx-prox-point}.
\begin{rem}
(Treating multiple sets in Theorem \ref{thm:cluster-pts-good}) We
remark that if we had analyzed Theorem \ref{thm:cluster-pts-good}
for the case where there are more than one indicator functions, then
analyzing \eqref{eq:perturbed-h} would mean having to introduce constraint
qualifications (for example, a Lipschitzian error bound assumption)
needed to deal with issues related to the stability of sets under
perturbations. 
\end{rem}
\end{rem}
\bibliographystyle{amsalpha}
\bibliography{../refs}

\end{document}